\documentclass[graybox]{svmult}

\usepackage{mathptmx}       
\usepackage{helvet}         
\usepackage{courier}        
\usepackage{type1cm}        

\usepackage{makeidx}         
\usepackage{graphicx}        
\usepackage{multicol}        
\usepackage[bottom]{footmisc}   

\makeindex     

\usepackage{mathrsfs}
\usepackage{amssymb}
\usepackage{amsmath}
\usepackage{amsfonts}

\newcommand{\cS}{{\mathcal{S}}}

\newcommand{\cX}{{\mathcal{X}}}
\newcommand{\cT}{{\mathcal{T}}}
\newcommand{\cD}{{\mathcal{D}}}
\newcommand{\cL}{{\mathcal{L}}}
\newcommand{\cR}{{\mathcal{R}}}
\newcommand{\cA}{{\mathcal{A}}}

\newcommand{\cK}{{\mathcal{K}}}
\newcommand{\cB}{{\mathcal{B}}}

\newcommand{\A}{\mathbb{A}}

\newcommand{\R}{\mathbb{R}}

\newcommand{\C}{\mathbb{C}}

\newcommand{\ind}{\text{ind}\,}
\newcommand{\nul}{\text{nul}\,}

\renewcommand{\Re}{\textrm{\rm Re}\,} 
\renewcommand{\Im}{\textrm{\rm Im}\,}

\newcommand{\sgn}{\mathrm{sgn}\,}
\newcommand{\ess}{\sigma_\mathrm{\tiny{ess}}}
\newcommand{\ptsp}{\sigma_\mathrm{\tiny{pt}}}

\begin{document}

\title*{Spectral stability of traveling fronts for nonlinear hyperbolic equations of bistable type}

\author{Corrado Lattanzio, Corrado Mascia, Ram\'{o}n G. Plaza and Chiara Simeoni}
\institute{Corrado Lattanzio \at Dipartimento di Ingegneria e Scienze dell'Informazione e Matematica, Universit\`a degli Studi 
dell'Aquila, via Vetoio, Coppito I-67100, L'Aquila (Italy), \email{corrado@univaq.it}
\and Corrado Mascia \at Dipartimento di Matematica ``Guido Castelnuovo'', Sapienza Universit\`a di Roma, P.le Aldo
Moro 2, I-00185 Roma (Italy), \email{mascia@mat.uniroma1.it}
\and Ram\'{o}n G. Plaza \at Instituto de Investigaciones en Matem\'aticas Aplicadas y en Sistemas, Universidad Nacional 
Aut\'onoma de M\'exico, Circuito Escolar s/n C.P. 04510 Cd.\ de M\'exico (Mexico), \email{plaza@mym.iimas.unam.mx}
\and Chiara Simeoni \at Laboratoire J.A. Dieudonn\'e UMR CNRS 7351, Universit\'e de Nice Sophia-Antipolis, Parc Valrose 
06108 Nice Cedex 02 (France), \email{simeoni@unice.fr}
}

\maketitle

\abstract*{This paper addresses the existence and spectral stability of traveling fronts for nonlinear hyperbolic equations with a positive ``damping" term and a reaction function of bistable type. Particular cases include the relaxed Allen-Cahn equation and the nonlinear version of the telegrapher's equation with bistable reaction term. The existence theory of the fronts is revisited, yielding useful properties such as exponential decay to the asymptotic rest states and a variational formula for the unique wave speed. The spectral problem associated to the linearized equation around the front is established. It is shown that the spectrum of the perturbation problem is stable, that is, it is located in the complex half plane with negative real part, with the exception of the eigenvalue zero associated to translation invariance, which is isolated and simple. In this fashion, it is shown that there exists an \textit{spectral gap} precluding the accumulation of essential spectrum near the origin. To show that the point spectrum is stable we introduce a transformation of the eigenfunctions that allows to employ energy estimates in the frequency regime. This method produces a new proof of equivalent results for the relaxed Allen-Cahn case and extends the former to a wider class of equations. This result is a first step in a more general program pertaining to the nonlinear stability of the fronts under small perturbations, a problem which remains open.}

\abstract{This paper addresses the existence and spectral stability of traveling fronts for nonlinear hyperbolic equations with a positive ``damping" term and a reaction function of bistable type. Particular cases of the former include the relaxed Allen-Cahn equation and the nonlinear version of the telegrapher's equation with bistable reaction term. The existence theory of the fronts is revisited, yielding useful properties such as exponential decay to the asymptotic rest states and a variational formula for the unique wave speed. The spectral problem associated to the linearized equation around the front is established. It is shown that the spectrum of the perturbation problem is stable, that is, it is located in the complex half plane with negative real part, with the exception of the eigenvalue zero associated to translation invariance, which is isolated and simple. In this fashion, it is shown that there exists an \textit{spectral gap} precluding the accumulation of essential spectrum near the origin. To show that the point spectrum is stable we introduce a transformation of the eigenfunctions that allows to employ energy estimates in the frequency regime. This method produces a new proof of equivalent results for the relaxed Allen-Cahn case and extends the former to a wider class of equations. This result is a first step in a more general program pertaining to the nonlinear stability of the fronts under small perturbations, a problem which remains open.}


\section{Introduction}

This paper studies the stability of traveling wave solutions to scalar hyperbolic equations of the form
\begin{equation}
\label{hypAC}
\tau u_{tt} + g(u,\tau) u_t = u_{xx} + f(u),
\end{equation}
where $u$ is a scalar, $x \in \R$, $t > 0$, and $\tau \geq 0$ is a constant. Note that \eqref{hypAC} is a nonlinear wave equation with a ``damping term", $g$, and a nonlinear reaction term $f$. Hyperbolic equations of this form often support traveling wave solutions, also called traveling fronts, which are special solutions describing coherent structures which propagate along a particular direction with a certain wave speed. In a previous contribution \cite{LMPS16}, we analyzed the existence and stability of propagating fronts for a one-dimensional model which is a particular case of equation \eqref{hypAC}, called the \textit{Allen-Cahn equation with relaxation}. The motivation for the present study is to explore both the existence and the stability of such configurations for a wider class of equations, which arises in other contexts.

We make the following assumptions. First, the reaction function $f : 
\R \to \R$ is supposed to be of \textit{bistable type}\footnote{also called of Nagumo \cite{NAY62,McKe70}, or 
Allen-Cahn \cite{AlCa79} type.}, that is, $f \in C^2([0,1];\R)$ has two stable 
equilibria at $u=0, u=1$, and one unstable equilibrium point at $u = \alpha \in (0,1)$, more precisely,
\begin{equation}
\label{H1}
\tag{H1}
	\begin{aligned}
	&f(0)=f(\alpha)=f(1)=0,
		&\qquad &f'(0), f'(1)<0,\quad f'(\alpha)>0,\\
	&f(u)>0\textrm{ for all } \, u \in(\alpha,1),
		&\qquad &f(u)<0\textrm{ for all } \, u \in (0,\alpha),
	\end{aligned}
\end{equation}
for a certain $\alpha \in (0,1)$. A well-known example is the widely used cubic polynomial
\begin{equation}
\label{cubicf}
	f(u)= u(1-u)(u-\alpha),
\end{equation}
with $\alpha \in (0,1)$.

Reaction functions of bistable type arise in many models of natural phenomena, such as kinetics of 
biomolecular reactions (cf. Mikha{\u\i}lov \cite{Mik94}), nerve conduction (see, e.g., 
Lieberstein \cite{Lbr67a}, McKean \cite{McKe70}) and electrothermal instability (cf. Iz\'us \textit{et al.} 
\cite{IDRWZB95}). In terms of continuous descriptions of the spread of biological populations, it is 
often applied to kinetics exhibiting positive growth rate for population densities over a threshold 
value ($u > \alpha$), and decay for densities below such value ($u < \alpha$). The latter is often 
described as the \textit{Allee effect}, in which aggregation can improve the survival rate of 
individuals (see Murray \cite{MurI3ed}).

Secondly, we are going to assume that the damping coefficient $g = g(u,\tau)$ in equation 
\eqref{hypAC} is regular enough and strictly positive. More precisely, we suppose that for some fixed value 
$\tau_m>0$, there holds
\begin{equation}
\label{H2}
\tag{H2}
g \in C^1(\R \times [0, \tau_m]), \;\; \; \text{and,} \quad \inf\left\{g(u, \tau) : u\in\R, 
\tau\in (0,\tau_m)\right\}\geq \delta_0>0,
\end{equation}
for some $\delta_0 > 0$ independent of $\tau_m$.

Assumption \eqref{H2} is an extension of the previously studied case of the Allen-Cahn model with relaxation 
\cite{LMPS16}, where
\begin{equation}
 \label{ACrelax}
 g(u,\tau) = 1 - \tau f'(u),
\end{equation}
and with $\tau > 0$ bounded above by the characteristic relaxation time associated to the reaction,
\[
 0 \leq \tau < \tau_m := \frac{1}{\max_{u \in [0,1]} |f'(u)|}.
\]
for which, clearly, $g(u,\tau) > 0$. If $F$ is an antiderivative such that $F'=-f$ with $F(0)=0$, that is,
\begin{equation*}
	F(u):=-\int_{0}^{u} f(v)\,dv,
\end{equation*}
then $F$ can be interpreted as the Allen-Cahn two-well potential (see Figure \ref{figpotAC}).

\begin{figure}[htb]
\begin{center}
\includegraphics[width=5.75cm]{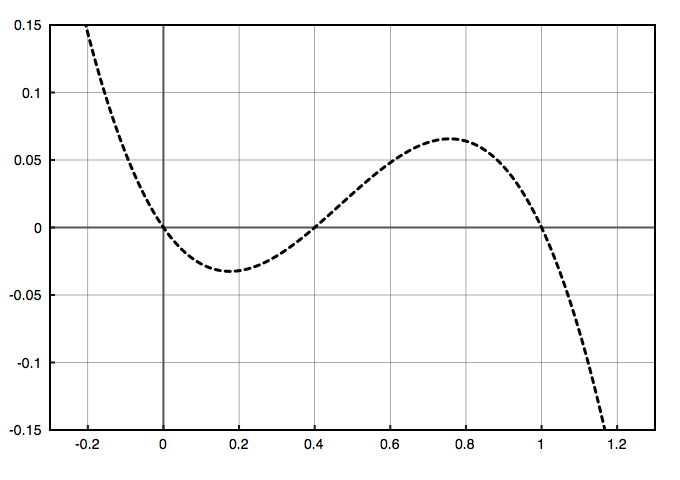}
\includegraphics[width=5.75cm]{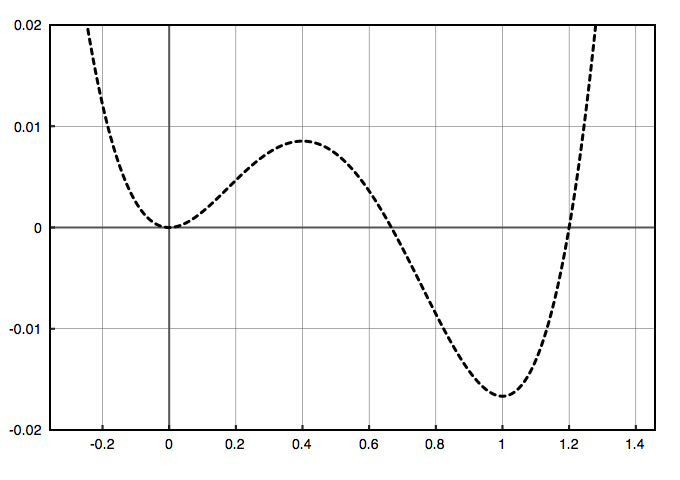} 
\end{center}
\caption{\footnotesize The bistable cubic function $f(u)=u(1-u)(u-0.4)$ (left)
and the corresponding two-well potential $F$ (right).\label{figpotAC}}
\end{figure}

Another example of interest is the \textit{nonlinear telegrapher's 
equation} \cite{Holm1}, where
\begin{equation}
 \label{telegr}
 g(u,\tau) \equiv 1,
\end{equation}
for all $u \in \R$ and $\tau \geq 0$.

\begin{remark}
There exist situations where the appearance of a diffusion coefficient $\varepsilon > 0$ in \eqref{hypAC},
\[
 \tau u_{tt} + g(u,\tau) u_t = \varepsilon u_{xx} + f(u),
\]
is important, for example, in the study of slow motion of solutions or their \textit{metastability} 
\cite{FLM17,Fol17}, when $0 < \varepsilon \ll 1$ is supposed to be small. For the problem of existence and stability of fronts, however, 
the size of 
$\varepsilon$ plays no role, and by rescaling the space variable, $x\mapsto x/\varepsilon$, we recover 
equation 
\eqref{hypAC}. Therefore, our analysis also applies to the more general model with arbitrary (constant) 
diffusion and we can work with equation \eqref{hypAC} directly without loss of generality.
\end{remark}

In this paper we establish the spectral stability of traveling fronts for \eqref{hypAC} under the sole structural assumptions \eqref{H1} and \eqref{H2}, which include many models in population dynamics, microstructures and relaxation mechanisms, among others. In Section \ref{secexist} we prove that traveling fronts exist and provide some of their more important features and properties. Section \ref{secperturb} contains the perturbation problem and describes how to formulate a natural spectral problem (after linearization of the equation around the front), whose analysis encodes the most fundamental stability properties. We show that there exists two different but equivalent ways to formulate the spectral problem. In Section \ref{secess} we analyze the asymptotic systems associated to the perturbed equations and locate the essential spectrum. Section \ref{secptsp} contains the proof that the point spectrum is stable (via energy estimates in the frequency regime), the simplicity of the eigenvalue zero associated to translation, as well as the statement of our main result (see Theorem \ref{mainthm}). Finally, in section \ref{secdisc} we make some concluding remarks.

\section{Structure of traveling fronts}
\label{secexist}

In this section we review the existence theory and structural properties of front solutions to equations of 
the form \eqref{hypAC}. In a recent contribution, Gilding and Kersner \cite{GiKe15} established the necessary 
and sufficient conditions for the existence of traveling wave solutions to equation \eqref{hypAC} with 
reaction function of bistable type under the assumption of positive damping $g > 0$. The authors make use of an 
integral equation approach. For completeness, in this section we present an existence result which applies a 
different technique based on the computation of the index of a rotating vector field of the dynamical system  
with respect to the velocity (in the sense of Perko \cite{Per1}); this proof resembles our previous analysis 
in the particular case of the relaxed Allen-Cahn model \cite{LMPS16}. With this approach we are able to 
derive further structural properties, such as the exponential decay of the solutions and a variational formula for the 
(unique) wave speed, which are not available from the integral formulation in \cite{GiKe15}.

\subsection{Existence}
\label{sect:existence}

We look for solutions to \eqref{hypAC} of the form
\begin{equation*}
	u(x,t)=U(\xi)\quad\textrm{with}\quad\xi=x-ct,
	\qquad\textrm{and}\qquad
	U(-\infty)=0,\quad
	U(+\infty)=1.
\end{equation*}
Substituting into \eqref{hypAC}, we obtain the equation
\begin{equation}
\label{twode0}
	(1-c^2\tau)U''+c\,g(U,\tau)U'+f(U)=0,
\end{equation}
where $' \, := d/d\xi$.

\begin{proposition}\label{prop:properties}
Let assumptions \eqref{H1} and \eqref{H2} be satisfied,
and let $U=U(\xi)$ be a solution to \eqref{twode0} together with the asymptotic
conditions $U(-\infty)=0$ and $U(+\infty)=1$. Then,
	\par
\textit{(i)} (speed sign) the velocity $c$ has the same sign of $-\int_{0}^{1} f(u)\,du$;\par
\textit{(ii)} (subcharacteristic condition) the velocity $c$ necessarily satisfies
\begin{equation}
 \label{subchar}
c^2\tau < 1.
\end{equation}
\end{proposition}

\begin{proof} 
\smartqed 
(i) Multiplying equation \eqref{twode0} by $U'$ and integrating in $\R$, we obtain
\begin{equation*}
	c\int_{\R} g(U,\tau)\left|U'\right|^2\,dx=F(1)-F(0).
\end{equation*}
where $F'=-f$. Thus, $\mathrm{sgn}(c) = \mathrm{sgn}(F(1) - F(0))$, as $g(U,\tau) > 0$.

(ii) The case $c = 0$ is manifest. If $c > 0$ then multiply equation \eqref{twode0} by $U'$. This yields,
\[
 (1-c^2 \tau)U'' U' + cg(U,\tau)|U'|^2 + f(U) U' = 0.
\]
Since $f = - F'$ last equation is equivalent to
\begin{equation}
 \label{eqfive}
 \Big( \tfrac{1}{2}(1-c^2 \tau) |U'|^2 - F(U) \Big)' + cg(U,\tau) |U'|^2 = 0. 
\end{equation}
Integrate equation \eqref{eqfive} in $(\xi,+\infty)$, to 
obtain
\begin{equation}
\label{launo}
\tfrac{1}{2}(1-c^2 \tau) |U'(\xi)|^2 = F(U(\xi)) - F(1) + c  \int_{\xi}^{+\infty} g(U(s),\tau)|U'(s)|^2 
\, ds,
\end{equation}
and choose $\xi \gg 1$, large enough so that $U(\xi) \in (\alpha, 1)$ (as $U(+\infty) = 1$).
Since $f(u) > 0$ for $u \in (\alpha, 1)$ and $U(\xi) \in (\alpha, 1)$, clearly
\[
F(U(\xi)) - F(1) = \int_{U(\xi)}^1 f(s) \, ds > 0.
\]
Since we are assuming $c > 0$ and since $g(U,\tau) > 0$, clearly the right hand side of \eqref{launo} is positive, yielding $1 > c^2 
\tau$. The case $c<0$ can be treated similarly.
\qed \end{proof}

\begin{remark}
Notice that if $F(0)=F(1)$, then the speed $c$ is necessarily zero and the equation for the profile reduces
to the one for traveling waves for the parabolic Allen-Cahn equation.
\end{remark}

We now prove an auxiliary result.

\begin{proposition}\label{prop:auxode}
Let assumptions \eqref{H1} - \eqref{H2} be satisfied.
Then there exists a unique value $\gamma\in\R$, denoted by $\gamma_\ast=\gamma_\ast(\tau)$,
such that the equation
\begin{equation}\label{auxode}
	V''+\gamma\,g(V,\tau)V'+f(V)=0
\end{equation}
has a monotone increasing solution, $V=V(\xi)$ with asymptotic limits $V(-\infty)=0$ and $V(+\infty)=1$.
\end{proposition}

The proof of Proposition \ref{prop:auxode} consists of showing that there exists a heteroclinic
connection between the singular points $(V,V')=(0,0)$ and $(V,V')=(1,0)$.
We follow a standard shooting argument starting from the local analysis near the asymptotic
states, and use the special dependence with respect to the parameter $\gamma$ to
show that there is a single value $\gamma_\ast$ for which there exists a connecting orbit.
The strategy closely resembles the one presented in H\"arterich and Mascia \cite{HaeMa1}.
For shortness, we drop the dependence of $g$ with respect to $\tau$.

\begin{proof}[of Proposition \ref{prop:auxode}]
\smartqed
The second order differential equation \eqref{auxode} can be rewritten as
\begin{equation}\label{firstorder}
	\left\{\begin{aligned}
		V'&=\Phi(V,W;\gamma):=W,\\
		W'&=\Psi(V,W;\gamma):=-f(V)-\gamma\,g(V)\,W,
	\end{aligned}\right.
\end{equation}
possessing the two singular points $(0,0)$ and $(1,0)$.

1. Linearizing at $(\bar u,0)$, we obtain the matrix 
\begin{equation*}
	\begin{pmatrix}
		\partial_V \Phi	&	&\partial_W \Phi\\
		\partial_V \Psi	&	&\partial_W \Psi
		\end{pmatrix}
	=\begin{pmatrix}
		0		&	&1\\
		-f'(\bar u)	& & \, -\gamma\,g(\bar u)
		\end{pmatrix}.
\end{equation*}
In particular, since $f'(0)$ and $f'(1)$ are negative, $(0,0)$ and $(1,0)$
are saddles for \eqref{firstorder}.
The positive eigenvalue $\mu^+_0$ at $(0,0)$ and the negative eigenvalue
$\mu^-_1$  at $(1,0)$ are 
\begin{equation*}
	\begin{aligned}
		\mu^+_0&=\frac12\left(\sqrt{(\gamma\,g(0))^2-4f'(0)}-\gamma\,g(0)\right),
			\\
		\mu^-_1&=-\frac12\left(\sqrt{(\gamma\,g(1))^2-4f'(1)}+\gamma\,g(1)\right).
	\end{aligned}
\end{equation*}

We denote by $\mathcal{U}_0(\gamma)$ the intersection of the unstable manifold of $(0,0)$
and the set $\{(V,W)\,:\,W>0\}$, and by $\mathcal{S}_1(\gamma)$ the intersection of the
unstable manifold of $(1,0)$ and the set $\{(V,W)\,:\,W>0\}$.

2. Let $\gamma<0$ and $\hat W>M/c_0|\gamma|$, where $M:=\max\{f(u)\,:\,u\in(\alpha,1)\}$.
The solution trajectory passing through $(\alpha,W_0)$ is the graph of the solution
$\omega=\omega(V)$ to the Cauchy problem
\begin{equation}\label{trajeq}
	\frac{d\omega}{dV}=-\frac{f(V)}{\omega}-\gamma\,g(V),
\end{equation}
with initial condition $\omega(\alpha)=\hat W$.
Denote its interval of maximal existence by $I$, and observe that $\omega$ is strictly increasing in $I\cap 
[\alpha,1]$.
Indeed, since $d\omega/dV(\alpha)=-\gamma\,g(\alpha)>0$, the function $\omega$ is strictly
increasing for $V\in(\alpha,\alpha+\delta)$ for some $\delta>0$.
Moreover, if $\omega>\hat W$ and $V\in[\alpha,1]$ there holds
\begin{equation*}
	\frac{d\omega}{dV}\geq -\frac{M}{\hat W}+c_0|\gamma|>0,
\end{equation*}
and the claim follows from a standard continuation argument.
As a consequence, the derivative of $\omega$ is a priori bounded and 
the interval $I$ contains the interval $[\alpha,1]$.

Since the vector field $(\Phi,\Psi)$ point downward along the segment $(\alpha,1)\times\{0\}$,
the curve $\mathcal{S}_1(\gamma)$ intersect the line $V=\alpha$ at some value $W_1(\gamma)\geq 0$
for $\gamma<0$. 
Similar arguments show that $\mathcal{U}_0(\gamma)$ intersects the line
$V=\alpha$ at some $W_0(\gamma)$ for $\gamma>0$. 

3. Since 
\begin{equation*}
	\det\begin{pmatrix}
		\Phi				& &\Psi\\
		\partial_\gamma\Phi	& & \partial_\gamma \Psi
		\end{pmatrix}
	=\det\begin{pmatrix}
		W	& & \, -f-\gamma\,g\,W\\
		0	& & \, -g\,W
		\end{pmatrix}=-g\,W^2\leq -c_0 W^2\leq 0,
\end{equation*}
the vector field defining the differential system is a \textit{rotated vector field}
with respect to the parameter $\gamma$ (see Perko \cite{Per1}).
As a consequence, the graphs $\mathcal{U}_0(\gamma)$ and $\mathcal{S}_1(\gamma)$
rotate clockwise as the parameter $\gamma$ increases.
Therefore, the map $W_{0}=W_{0}(\gamma)$ is monotone decreasing in $(0,+\infty)$
and the map $W_{1}(\gamma)=W_{1}(\gamma)$ is monotone increasing in $(-\infty,0)$.

4. If $\bar V$ is a relative maximum point for a solution $\omega$ to \eqref{trajeq},
then
\begin{equation*}
	|\omega(\bar V)|=\frac{|f(\bar V)|}{|\gamma|\,g(\bar V)}\leq \frac{M}{c_0|\gamma|},
\end{equation*}
where $M$ is the maximum of $|f|$ in $(0,1)$.
Thus, $W_0(\gamma)\to 0$ as $\gamma\to+\infty$ 
and $W_1(\gamma)\to 0$ as $\gamma\to-\infty$.
Following Hadeler \cite{Had1}, let us note that one can also prove that there exist values
$\gamma_\pm$ with $\gamma_0<0<\gamma_1$ such that $W_0(\gamma_0)=0$ and
$W_1(\gamma_1)=0$.
Then, for any $\gamma\geq \gamma_0$ the trajectory $\mathcal{U}_0(\gamma)$
describes a heteroclinic connection between $0$ and $\alpha$;
similarly, for any $\gamma\leq \gamma_1$ the trajectory $\mathcal{S}_1(\gamma)$
describes a heteroclinic connection between $\alpha$ and $1$.

5. From monotonicity of $W_{0}$ and $W_{1}$, we infer that they both have limits as $\gamma\to 0$.
Additionally, the trajectory equation \eqref{trajeq} shows that such limiting values $W_0(0)$ and
$W_{1}(0)$ are finite and can be computed explictly, taking advantage of the conserved quantity
$W^2-2F(V)$, yielding
\begin{equation*}
	W_0(0)=\sqrt{2\bigl(F(\alpha)-F(0)\bigr)}
		\qquad\textrm{and}\qquad
	W_1(0)=\sqrt{2\bigl(F(\alpha)-F(1)\bigr)}.
\end{equation*}
Since the solution depends continuously with respect to the parameter $\gamma$, 
there exist $\gamma_0, \gamma_1$ with $-\infty\leq \gamma_0<0<\gamma_1\leq +\infty$,
such that $W_0$ is defined (and monotone decreasing) in $(\gamma_0,+\infty)$ and $W_1$
is defined  (and monotone increasing) in $(-\infty,\gamma_1)$.
If $\gamma_0$ is finite, $W_0\to+\infty$ as $\gamma\to\gamma_0^+$;
similarly, if $\gamma_1$ is finite, $W_1\to+\infty$ as $\gamma\to\gamma_1^-$.
\begin{figure}[htb]
\begin{center}
\includegraphics[width=9cm, height=6.5cm]{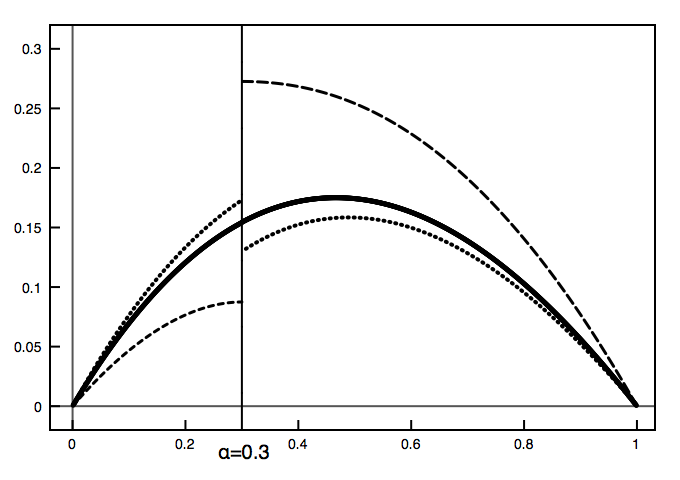}
\end{center}
\caption{\footnotesize Graphs of the curves $\mathcal{U}_0(\gamma)$ (left) and $\mathcal{S}_1(\gamma)$ (right) 
in the plane $(V,W)$ for different values of $\gamma$:
dashed line $\gamma=0$; dotted line $\gamma=-0.40$; and continuous line $\gamma=-0.32$. Here we considered the case of a cubic nonlinearity $f(u)=u(1-u)(u-\alpha)$ with $\alpha=0.3$, and damping term of Cattaneo-Maxwell type, $g(u;\tau)=1-\tau f'(u)$, where $\tau=1$.
}
\end{figure}

6. Let us consider the difference function $h:=W_1-W_0$ defined in $(\gamma_0,\gamma_1)$.
As a consequence of the properties of $W_0$ and $W_1$, we infer that $h$ is continuous,
monotone increasing and such that
\begin{equation*}
	\liminf_{\gamma\to\gamma_0^+} h(\gamma)<0,\quad
	\liminf_{\gamma\to\gamma_1^-} h(\gamma)>0.
\end{equation*}
In particular, there exists a unique $\gamma_\ast$ such that
$W_0(\gamma_\ast)=W_1(\gamma_\ast)$.
For such critical value, the conjuction of the curves $\mathcal{U}_0(\gamma_\ast)$ and
$\mathcal{S}_1(\gamma_\ast)$ gives the desired connection.
Uniqueness of the wave speed $\gamma_\ast$ follows from the monotonicity
of  the functions $W_0$ and $W_1$.
\qed \end{proof}

\begin{remark}
Equation \eqref{auxode} arises also in the case of  reaction-diffusion equations
with density-dependent diffusion
\begin{equation*}
w_t = \varphi(w)_{xx} + f(w),
\end{equation*}
where $\varphi$ is a strictly increasing function. Inserting the traveling wave profile \textit{ansatz} 
$w(x,t)=W(x-\gamma t)$ and setting $V:=\phi(W)$ yields
\begin{equation*}
	\frac{d^2V}{d\eta^2}+\gamma \psi'(V)\frac{dV}{d\eta}+f(\psi(V))=0,
\end{equation*}
where $\psi$ is the inverse function of $\phi$.
In fact, existence of heteroclinic solutions for \eqref{auxode} could be also proved by
appropriately changing the dependent variable $V$ 
and applying the general result proved by Engler \cite{Eng2} that relates the existence
of traveling wave solutions of reaction-diffusion equations with constant diffusion
coefficient to the ones of the density-dependent diffusion coefficient case. 
\end{remark}

\begin{example}
\label{exAllenCahn}
In the special case of a nonlinear telegrapher's equation with cubic reaction function, namely,
\begin{equation}\label{fbist_g1}
	g(u,\tau)=1,\qquad f(u)=\kappa\,u(1-u)(u-\alpha),
\end{equation}
we can look for $W=V'$ with the form $W(V) = AV(1-V)$, where $A$ is a constant to be determined.
Inserting in \eqref{auxode}, we deduce the following constraints on $A$ and $\gamma$
\begin{equation*}
	A^2+\gamma\,A-\kappa\,\alpha=0,\qquad 2A^2-\kappa=0,
\end{equation*}
giving the explicit formulas $A=\sqrt{\kappa/2}$ and
\begin{equation}\label{gast_fbist_g1}
	\gamma_\ast=\gamma_{{}_{\textrm{\tiny AC}}}:=\sqrt{\frac{2}{\kappa}}\left(\alpha-\frac12\right),
\end{equation}
which is the speed of propagation for the (parabolic) Allen--Cahn equation.
In the significant relaxation case $g(u,\tau)=1-\tau f'(u)$, the same simplification does not hold and
an analogous explicit formula for the critical speed $\gamma_\ast$ is not available. However, as in the case 
of the standard Allen--Cahn equation, it is possible to establish
a min-max variational characterization for the critical speed $\gamma_\ast$ (cf. Hamel \cite{Hame99}; see also 
\cite{MFF1}).
\end{example}

\begin{proposition}[variational formula for the speed]
\label{prop:minmax}
Let assumptions \eqref{H1} - \eqref{H2} be satisfied.
Set
\begin{equation*}
	\mathcal{W}:=\{W\in C^2(\R)\,:\, W(x)\in(0,1),\, W'(x)>0\;\,\text{for any}\, \;x\in\R\}.
\end{equation*}
Then the speed $\gamma_\ast$ defined in Proposition \ref{prop:auxode} is such that
\begin{equation}\label{minmax}
	\gamma_\ast=-\inf_{W\in \mathcal W} \sup_{x\in\R}\frac{W''+f(W)}{g(W)\,W'}
			=-\sup_{W\in \mathcal W} \inf_{x\in\R}\frac{W''+f(W)}{g(W)\,W'}.
\end{equation}
\end{proposition}

\begin{proof} 
\smartqed
We give a sketch of the proof.
Denote by $V$ the traveling profile given by Proposition \ref{prop:auxode};
then there holds $\gamma_\ast=-(V''+f(V))/g(V)V'$.
Since $V\in\mathcal W$, we infer the inequalities 
\begin{equation*}
	 \underline{\gamma}:= \inf_{W\in \mathcal W} \sup_{x\in\R}\frac{-\bigl(W''+f(W)\bigr)}{g(W)\,W'}
	\leq \gamma_\ast
	\leq \overline{\gamma}:=\sup_{W\in \mathcal W} \inf_{x\in\R}\frac{-\bigl(W''+f(W)\bigr)}{g(W)\,W'}.
\end{equation*}
If $\gamma_\ast<\overline{\gamma}$ then for any $\gamma\in(\gamma_\ast,\overline{\gamma})$,
there exists a function $W\in\mathcal W$ such that
\begin{equation*}
 \inf_{x\in\R}\frac{-\bigl(W''+f(W)\bigr)}{g(W)\,W'}\geq \gamma.
\end{equation*}
As a consequence, we deduce
\begin{equation*}
	W''+\gamma\,g(W)\,W'+f(W)\leq 0\leq 
	(\gamma-\gamma_\ast) g(V)\,V'=V''+\gamma\,g(V)\,V'+f(V),
\end{equation*}
showing that $W$ and $U$ are, respectively, super- and subsolution for
\begin{equation}\label{elliptic}
	U''+\gamma\,g(U)\,U'+f(U)=0.
\end{equation}
Invoking a monotonicity argument \cite{ProWe84}, we deduce the existence of a solution 
$U$ to \eqref{elliptic} 
such that $V\leq U\leq W$,
thus satisfying, in particular, the asymptotic conditions $U(-\infty)=0$ and 
$U(+\infty)=1$.
Such statement contradicts the uniqueness of the speed $\gamma_\ast$
given in Proposition \ref{prop:auxode}.
Thus, $\gamma_\ast=\overline{\gamma}$. Proving in an analogous manner the equality 
$\gamma_\ast=\underline{\gamma}$,
we deduce formula \eqref{minmax}.
\qed \end{proof}

Independently from the variational characterization of the wave speed, the existence of a solution
for \eqref{twode0} with appropriate asymptotic values is a straightforward
consequence of Proposition \ref{prop:auxode}.
The relation between the speed $\gamma$ of Proposition \ref{prop:auxode} and $c$ for
\eqref{twode0} guarantees the uniqueness of the speed for the hyperbolic
Allen--Cahn equation.

\begin{theorem}[existence of a traveling front]
\label{theoexists}
Under assumptions \eqref{H1} - \eqref{H2} there exists a unique value $c\in\R$, denoted by 
$c_\ast=c_\ast(\tau)$,
such that the equation
\begin{equation}\label{twode}
	(1-c^2\tau)U''+c\,g(U,\tau)U'+f(U)=0.
\end{equation}
has a monotone increasing front solution $U=U(\xi)$ with $U(-\infty)=0$ and $U(+\infty)=1$. The value 
$c_\ast=c_\ast(\tau)$ is related to $\gamma_\ast=\gamma_\ast(\tau)$
of Proposition \ref{prop:auxode} by the relation
\begin{equation}\label{castgast}
	c_\ast=\frac{\gamma_\ast}{\sqrt{1+\tau\,\gamma_\ast^2}}.
\end{equation}
\end{theorem}

\begin{proof} 
\smartqed
Thanks to the subcharacteristic condition \eqref{subchar}, we can 
restrict our attention to
$c\in(-1/\sqrt{\tau},1/\sqrt{\tau})$.
By applying the change of variables
\begin{equation*}
	\sqrt{1-c^2\tau}\frac{d}{d\xi}=\frac{d}{d\eta},
\end{equation*}
and setting $\gamma=\gamma(c)=c/\sqrt{1-c^2\tau}$,
equation \eqref{twode} transforms into \eqref{auxode}.
Then the profile existence and uniqueness statement follows since
$\gamma=\gamma(c)$ is increasing and $\gamma(\pm 1 /\sqrt{\tau})=\pm\infty$. 
Relation \eqref{castgast} is obtained by inverting the function $\gamma = \gamma(c)$. 
\qed \end{proof}

\subsection{Exponential decay}\label{secexpdecay}

As a consequence of the analysis in Proposition \ref{prop:auxode} and Theorem \ref{theoexists}, the profile function decays to its asymptotic limits exponentially fast.

\begin{lemma}[exponential decay of the profile]
\label{lemexpdecay}
For each $\tau \geq 0$ the front solution and its derivatives satisfy
\begin{equation}
\label{expdecay}
|\partial_\xi^j(U(\xi) - U_\pm)| \leq C e^{-\eta|\xi|}, 
\end{equation}
for all $\xi \in \R$, $j = 0,1,2$, with uniform constants $C > 0$ and $\eta > 0$.
\end{lemma}
\begin{proof}
\smartqed
Suppose that $U = U(\xi)$ is the profile function of Theorem \ref{theoexists}, traveling with speed $c = c_*(\tau)$. As before,  $\xi = x - ct$ and $\, ' = d/d\xi$. If we denote $V = U'$ then $(U,V) = (U,V)(\xi)$ is an heteroclinic connection between the rest points
\[
(U_+, V_+) = (1,0) \quad \text{and,} \quad (U_-,V_-) = (0,0),
\]
as $\xi \to \pm \infty$, of the first order system
\begin{equation}
\label{firstODE}
\begin{pmatrix}U \\ V
\end{pmatrix}' = \begin{pmatrix} V \\ - (1-c^2 \tau)^{-1} (f(U) +c g(U,\tau)V)
\end{pmatrix} =: \begin{pmatrix} \hat{\Phi} \\ \hat{\Psi}\end{pmatrix} (U,V).
\end{equation} 
Linearizing around the asymptotic rest states we obtain
\[
\frac{D(\hat{\Phi}, \hat{\Psi})}{D(U,V)}(U_\pm, V_\pm) =  \begin{pmatrix} 0 & & 1 \\ (1-c^2\tau)^{-1} |a_\pm| & & \, -(1-c^2\tau)^{-1}c b_\pm \end{pmatrix},
\]
where, in view of assumptions \eqref{H1} and \eqref{H2}, we have denoted $a_\pm = f'(U_\pm) < 0$ and $b_\pm = g(U_\pm,\tau) > 0$. Its eigenvalues are
\[
\mu_{1,2}^\pm = - \tfrac{1}{2} cb_\pm (1-c^2 \tau)^{-1} \pm \tfrac{1}{2} \sqrt{c^2 b_{\pm}^2(1-c^2 
\tau)^{-2} + 4(1-c^2 \tau)^{-1}|a_\pm|},
\]
which are real and the asymptotic states are non-degenerate hyperbolic points. The positive eigenvalue at  $(U_-,V_-) = (0,0)$ is
\[
\mu_2^- = - \tfrac{1}{2}cb_-(1-c^2 \tau)^{-1} + \tfrac{1}{2} \sqrt{c^2 b_{-}^2(1-c^2 \tau)^{-2} + 
4(1-c^2 \tau)^{-1}|a_-|},
\]
and the orbit decays to $(U_-,V_-) = (0,0)$ with exponental rate $|(U,V)(\xi)| \leq C e^{\mu_2^- \xi}$ as 
$\xi \to -\infty$ for some uniform $C > 0$.  The negative eigenvalue at $(U_+,V_+) = (1,0)$ is
\[
\mu_1^+ =  - \tfrac{1}{2}cb_+(1-c^2 \tau)^{-1} - \tfrac{1}{2} \sqrt{c^2 b_{+}^2(1-c^2 \tau)^{-2} + 
4(1-c^2 \tau)^{-1}|a_+|}, 
\]
and the orbit decays as $|(U,V)(\xi) - (1,0)| \leq C e^{-|\mu_1^+|\xi}$, when $\xi \to +\infty$. Thus, if we define $\eta = \min \{\mu_2^-, |\mu_1^+|\} > 0$ we obtain the result. Notice that $\eta = \eta(\tau) > 0$ for each fixed $\tau \geq 0$ and that $V' = U''$ also decays exponentially fast.
\qed
\end{proof}

\section{Perturbation equations and the stability problem}

\label{secperturb}

In this section we derive the equation for a perturbation of the traveling front, linearize it around the 
wave, and set up the associated spectral 
problem. 

For fixed $\tau > 0$ let $c = c_*(\tau) \in (-1/\sqrt{\tau},1/\sqrt{\tau})$ be the unique wave speed of the 
traveling front of Theorem \ref{theoexists}. We then recast equation \eqref{hypAC} in the moving coordinate 
frame and, with a slight abuse of notation, make the transformation $x \to x-ct$ so that the model equation 
\eqref{hypAC} now reads
\begin{equation}
 \label{newhypAC}
 \tau u_{tt} - 2c\tau u_{xt} + g(u,\tau) u_t = (1-c^2 \tau) u_{xx} + c g(u,\tau) u_x + f(u).
\end{equation}
From this point on and for the rest of the paper $x$ will denote the (Galilean) moving variable and 
the front profile $U = U(x)$ is now a stationary solution to \eqref{newhypAC}, satisfying
\begin{equation}
 \label{nprofileq}
 (1-c^2\tau) U_{xx} + cg(U,\tau) U_x + f(U) = 0.
\end{equation}
As before, the asymptotic limits are $U_+ = U(+\infty) = 1$ and $U_- = U(-\infty) = 0$. In view of Lemma 
\ref{lemexpdecay} the convergence of $U$ to its asymptotic limits is exponential,
\begin{equation}
 \label{expdec}
 |\partial_x^j(U - U_\pm) (x)| \leq C e^{- \eta |x|},
\end{equation}
as $x \to \pm \infty$ and for some $C, \eta > 0$.

\begin{remark}
\label{remUreg}
By regularity of the profile and its exponential decay, it is clear that $U_x \in H^1(\R)$. 
Apply a bootstrapping argument to verify that, in fact, $U_x \in H^3(\R)$. Details are left to the reader.
\end{remark}

\subsection{Equations for the perturbation and the spectral problem}

Let us consider solutions to \eqref{newhypAC} of the form $u(x,t) + U(x)$, where now $u = u(x,t)$ stands for 
a perturbation of the front. Upon substitution, we obtain the following nonlinear equation for the 
perturbation,
\begin{equation}
 \label{nlpert}
 \begin{aligned}
  \tau u_{tt} - 2c\tau u_{xt} &+ g(u+U,\tau) u_t = \\ &=(1-c^2 \tau) u_{xx} + (1-c^2 \tau) U_{xx} + c 
g(u+U,\tau)( 
u_x + U_x) + f(u).
 \end{aligned}
\end{equation}
Expand the nonlinear terms in Taylor series around $U$ and use the profile equation \eqref{nprofileq} to 
write equation \eqref{nlpert} as
\[
\begin{aligned}
 \tau u_{tt} - 2c\tau u_{xt} + g(U,\tau) u_t &= (1-c^2 \tau) u_{xx} + c g(U,\tau) u_x  + (c 
g_u(U,\tau)U_x + f'(U))u +\\ & \, + O(|uu_t|) + O(|uu_x|) + O(|u|^2).
\end{aligned}
\]
Let us define
\[
 a(x) := c g(U,\tau)_x + f'(U), \quad b(x) := g(U,\tau) \, > \, 0.
\]
Dropping the nonlinear terms we arrive at the following linearized equation for the perturbation
\begin{equation}
 \label{linequ}
 \tau u_{tt} - 2c\tau u_{xt} + b(x) u_t = (1-c^2 \tau) u_{xx} + c b(x) u_x  + a(x) u.
\end{equation}

Let us specialize the linear problem to solutions of the form $u(x,t) = e^{\lambda t} v(x)$, where $\lambda 
\in \C$ is the spectral parameter and $v$ belongs to an appropriate Banach space $X$. The result is the 
following spectral equation for $v$,
\begin{equation}
 \label{specprob}
 \lambda^2 \tau v - 2c \lambda \tau v_x + \lambda b(x) v = (1-c^2 \tau) v_{xx} + cb(x) v_x + a(x) v,
\end{equation}
for some $v \in X$, $\lambda \in \C$.

In this analysis we choose the perturbation space to be $X = L^2(\R;\C)$, and the domain of solutions to 
\eqref{specprob} to be $\mathcal{D} = H^2(\R;\C)$. In the sequel, $L^2$ and $H^m$, with $m > 0$, will denote the complex spaces $L^2(\R;\C)$ and $H^m(\R;\C)$, respectively, except where it is explicitly stated otherwise.

\begin{remark}
\label{rempencil}
 Notice that the spectral equation \eqref{specprob} is quadratic in $\lambda$. Under the substitution $\lambda 
= i \zeta$ equation \eqref{specprob} can be written in terms of a 
\textit{quadratic operator pencil} $\tilde \cA(\zeta)$ (cf. Markus \cite{Markus88}), given by 
\[
 \tilde \cA (\zeta) = \tilde \cA_0 + \zeta \tilde \cA_1 + \zeta^2 \tilde \cA_2,
\]
with
\[
 \begin{aligned}
   \tilde \cA_0 &= (1-c^2 \tau) \frac{d^2}{dx^2} + cb(x) \frac{d}{dx} + a(x),\\
   \tilde \cA_1 &= i 2 \tau \frac{d}{dx} - ib(x),\\
   \tilde \cA_2 &= \tau. 
 \end{aligned}
\]
It is easy to see that \eqref{specprob} is equivalent to $\tilde \cA(\zeta) v = 0$. The transformation
$v_1 = v$, $v_2 = \lambda v - cv_x$ defines an 
appropriate Cartesian product of the base space which allows us to write equation \eqref{specprob} as a 
genuine eigenvalue problem in the form
\begin{equation}
\label{defcLtau}
 \lambda \begin{pmatrix}
          v_1 \\ v_2
         \end{pmatrix} = \begin{pmatrix}
          c \partial_x & & 1 \\ \tau^{-1} (\partial_x^2 + a(x)) & & \, c \partial_x - \tau^{-1}b(x)
         \end{pmatrix}\begin{pmatrix}v_1 \\ v_2 \end{pmatrix} =: \cL^\tau \begin{pmatrix}v_1 \\ 
v_2\end{pmatrix}.
\end{equation}
The linear operator $\cL^\tau$ (densely defined in $L^2 \times L^2$ with domain 
$\cD(\cL^\tau) = H^2 \times H^1$ for $\tau > 0$) is often called the \textit{companion matrix} to the pencil $\tilde 
\cA$ (see \cite{BrJoK14,KoMi14,LaSu1} for further information). 
\end{remark}

\subsection{Reformulation as a first order system}

According to custom in the literature of stability of nonlinear waves \cite{AGJ90,KaPro13}, we now recast the 
spectral problem \eqref{specprob} as a first order system in the frequency regime of the form
\begin{equation} 
\label{Wsystem}
W_x = \A^\tau(x,\lambda) W,
\end{equation}
where $\lambda \in \C$ is a parameter and $\tau > 0$ is fixed. Indeed, making
\[
 W= \begin{pmatrix}
     v \\ v_x
    \end{pmatrix},
\]
and noticing that because of the subcharacteristic condition (see Proposition \ref{prop:properties} 
(ii)) there holds $1 - c^2 \tau > 0$, we obtain a first order ODE system of the form \eqref{Wsystem} 
with coefficient matrix given by
\begin{equation}
 \label{coeffA}
\A^\tau(x,\lambda) = (1 - c^2 \tau)^{-1}\begin{pmatrix}
                                        0 &  & 1-c^2\tau \\ \tau \lambda^2 + \lambda b(x) - a(x) &  & \, -c(b(x) + 2 
\tau \lambda)
                                       \end{pmatrix}.
\end{equation}
Since $U(x) \to U_\pm$ as $x \to \pm \infty$, with $U_- = 0$, $U_+ = 1$, let us denote
\[
\begin{aligned}
a_\pm &= \lim_{x \to \pm \infty} a(x) = \lim_{x \to \pm \infty} \big( f'(U) + g_u(U,\tau)U_x \big) = 
f'(U_\pm) 
< 0,\\
b_\pm &= \lim_{x \to \pm \infty} b(x) = \lim_{x \to \pm \infty} g(U,\tau) = g(U_\pm,\tau) > 0,
\end{aligned}
\]
because $U_x \to 0$, $f'(1)$, $f'(0) < 0$ and $g(U,\tau) > 0$, by hypotheses \eqref{H1} and \eqref{H2}. In 
this fashion, we denote the asymptotic coefficient matrices as
\begin{equation}
\label{asympcoeff}
\begin{aligned}
\A^\tau_\pm (\lambda) &:= \lim_{x \to \pm \infty} \A^\tau(x,\lambda) \\&= (1 - c^2\tau)^{-1} \begin{pmatrix}
                                                                                         0 &  & 1 - c^2 \tau \\ 
\tau \lambda^2 + \lambda b_\pm + |a_\pm| &  & \, -c(b_\pm + 2 \tau \lambda)
                                                                                        \end{pmatrix},
\end{aligned}
\end{equation}
for each $\tau \geq 0$, $\lambda \in \C$. 

It is convenient to define the spectra and 
resolvent of the spectral problem \eqref{specprob} in terms of the first order systems \eqref{Wsystem}. 
Consider the following family of linear, 
closed, densely defined operators
\[
 \cT^{\tau}(\lambda) : \bar \cD \to L^2 \times L^2,
\]
\[
 \cT^{\tau}(\lambda) := \partial_x - \A^{\tau}(x,\lambda),
\]
with domain $\bar \cD = H^1 \times H^1$, indexed by $\tau \geq 0$ and parametrized by $\lambda \in \C$. With 
a 
slight abuse of notation we call $W \in H^1 \times H^1$ an \textit{eigenfunction} associated to the 
eigenvalue 
$\lambda \in \C$ provided $W$ is a bounded solution to the equation
\[
 \cT^{\tau}(\lambda) W = W_x - \A^{\tau}(x,\lambda) W = 0.
\]

\begin{definition}[resolvent and spectra]
\label{defsigmatwo}
For fixed $\tau \geq 0$ we define,
\[
\begin{aligned}
\rho &:= \{\lambda \in \C \, : \, \cT^{\tau}(\lambda) \,\text{ is injective and onto, and } 
\cT^{\tau}(\lambda)^{-1} \, \text{is bounded} \, \},\\
\ptsp &:= \{ \lambda \in \C\,: \; \cT^{\tau}(\lambda) \,\text{ is Fredholm with index zero and has a} \\
& \qquad \qquad \qquad \text{non-trivial kernel} \},\\
\ess &:= \{ \lambda \in \C\,: \; \cT^{\tau}(\lambda) \,\text{ is either not Fredholm or has index 
different } \\
& \qquad \qquad \qquad \text{from zero} \}.
\end{aligned}
\]
The spectrum $\sigma$ of problem \eqref{specprob} is defined as $\sigma = \ess \cup 
\ptsp$. Since $\cT^{\tau}(\lambda)$ is closed, we know that $\rho = \C \backslash \sigma$ (cf. 
Kato \cite{Kat80}).
\end{definition}

\begin{remark}
 This definition of spectrum is due to Weyl \cite{We10}, making $\ess$ a large set but easy to compute, 
whereas $\ptsp$ is a discrete set of isolated eigenvalues with finite multiplicity (see Remark 2.2.4 
in \cite{KaPro13}). We remind the reader that a closed operator $\mathcal{L}$ is said to be 
Fredholm if its range $\mathcal{R(L)}$ is closed, and both its nullity, $\nul\mathcal{L} = \dim \ker 
\mathcal{L}$, and its deficiency, $\mathrm{def} \,\mathcal{L} = \mathrm{codim} \, \mathcal{R(L)}$, are 
finite. In such a case the index of $\cL$ is defined as
$\ind \mathcal{L} = \nul \mathcal{L} - \mathrm{def} \, \mathcal{L}$ (cf. \cite{Kat80}).
\end{remark}

For each $\tau \geq 0$ we can write the coefficients as
\[
 \A^\tau(x,\lambda) = \A^\tau_0(x) + \lambda \A^\tau_1(x) + \lambda^2 \A^\tau_2(x),
\]
where
\[
 \A^\tau_0(x) = (1-c^2\tau)^{-1}\begin{pmatrix}
                 0 & & 1 - c^2\tau \\ -a(x) & & -cb(x)
                \end{pmatrix},
\]
\[
 \A^\tau_1(x) = (1-c^2\tau)^{-1}\begin{pmatrix}
                                 0 & & 0 \\ b(x) & & -2c\tau 
                                \end{pmatrix},
\]
\[
 \A^\tau_2(x) = (1-c^2\tau)^{-1}\begin{pmatrix}
                                 0 & & 0 \\ \tau & & 0
                                \end{pmatrix}.
\]

Therefore, we may compute
\begin{equation}
 \label{derivlamA}
\partial_\lambda \A^\tau(x,\lambda) = \A^\tau_1(x) + 2 \lambda \A^\tau_2(x).
\end{equation}
Furthermore, if we regard the coefficients \eqref{coeffA} as functions from $(\lambda,\tau)$ into $L^\infty$ 
then they are analytic in $\lambda$ (quadratic polynomial) and continuous in $\tau$.

We also define the algebraic and geometric multiplicities of the elements in the point spectrum as 
follows.
\begin{definition}
 \label{defmult}
 Assume $\lambda \in \ptsp$. Its geometric multiplicity (\textit{g.m.}) is the maximal number of linearly 
independent elements in $\ker \cT^{\tau}(\lambda)$. Suppose $\lambda \in \ptsp$ has $g.m. = 1$, so that 
$\ker \cT^{\tau}(\lambda) =$ span $\{W_0\}$. We say $\lambda$ has algebraic multiplicity (\textit{a.m.}) 
equal to $m$ if we can solve
\[
 \cT^{\tau}(\lambda) W_j = \partial_\lambda \A^\tau(x,\lambda) W_{j-1},
\]
for each $j = 1, \ldots, m-1$, with $W_j \in H^1$, but there is no bounded $H^1$ solution $W$ to
\[
 \cT^{\tau}(\lambda) W = \partial_\lambda \A^\tau(x,\lambda)  W_{m-1}.
\]
For an arbitrary eigenvalue $\lambda \in \ptsp$ with $g.m.= l$, the algebraic multiplicity is defined as the 
sum of the multiplicities $\sum_k^l m_k$ of a maximal set of linearly independent elements in $\ker 
\cT^{\tau}(\lambda) = $ span $\{W_1, \ldots, W_l\}$.
\end{definition}

\begin{remark}
Notice that, unlike the operator defined in \eqref{defcLtau}, the spectral problem formulated as a first order system is well defined also for $\tau = 0$, as
\begin{equation}
\label{coeffA0}
\A^0(x,\lambda) = \begin{pmatrix}
                   0 & & 1 \\ \lambda b(x) - a(x) & & \, -c b(x)
                  \end{pmatrix},
\end{equation}
where the coefficients $a(x) = f'(U) + g(U,0)_x$, $b(x) = g(U,0)$ and the speed $c = c(0)$ are evaluated at 
$\tau = 0$. 
\end{remark}

Finally we remark that, due to translation invariance, $\lambda = 0$ belongs to the point spectrum.
\begin{lemma}
\label{lemzeroeigenv}
For each $\tau \geq 0$, $0 \in \ptsp$, with associated eigenfunction $\Phi = (U_x, U_{xx})^\top \in H^1 
\times H^1$.
\end{lemma}
\begin{proof} 
\smartqed
Follows by a direct calculation using the profile equation \eqref{nprofileq}. 
Notice that $U_{x} \in H^2$ (see Remark \ref{remUreg}), so that $\Phi = (U_x, U_{xx})^\top \in \ker 
\cT^\tau(0) \subset H^1 \times H^1$ is indeed an eigenfunction.
\qed \end{proof}

\subsection{Spectral equivalence}
\label{secequiv}

The seasoned reader might rightfully ask what is the relation between the spectrum of Definition \ref{defsigmatwo}, and the standard spectrum of the family of operators $\cL^\tau$ defined in \eqref{defcLtau} (see Remark \ref{rempencil}). Just like in the relaxed Allen-Cahn case (see Section 3 of \cite{LMPS16}), we shall prove that there is a one-to-one correspondence between the two sets, both in location and in multiplicities.

First observe that the family of operators $\cL^\tau$ in \eqref{defcLtau} is defined for parameter values of $\tau > 0$ only, whereas the first order systems \eqref{Wsystem} are well defined for $\tau = 0$ as well. (This happens because the hyperbolic equation \eqref{hypAC} actually degenerates into a parabolic equation when $\tau \to 0^+$.) Thus, we shall prove the spectral equivalence between the two spectral problems assuming that $\tau > 0$. Notice that for each $\tau > 0$ the operator $\cL^\tau : L^2 \times L^2 \to L^2 \times L^2$  is a closed, densely defined linear operator with domain $\cD(\cL^\tau) = H^2 \times H^1$.

\begin{lemma}
\label{lemQinv}
For each $\lambda \in \C$ and $\tau > 0$, the mapping
\[
\begin{aligned}
\cK : \ker (\cL^\tau - \lambda) &\subset H^2 \times H^1 \, \longrightarrow \ker \cT^\tau(\lambda)  \subset H^1 \times H^1,\\
\cK \begin{pmatrix} v_1 \\ v_2 \end{pmatrix} &:= \begin{pmatrix} v_1 \\ \partial_x v_1 \end{pmatrix}, \qquad \begin{pmatrix} v_1 \\ v_2 \end{pmatrix} \in \ker (\cL^\tau - \lambda),
\end{aligned}
\]
is one-to-one and onto.
\end{lemma}
\begin{proof}
\smartqed
First we check that $(v_1,v_2)^\top \in \ker (\cL^\tau - \lambda)$ implies that $\cK (v_1,v_2)^\top \in \ker \cT^\tau(\lambda)$. In that case we have the system
\[
\begin{aligned}
c \partial_x v_1 + v_2 &= \lambda v_1 \\
\tau^{-1} (\partial_x^2 + a(x)) v_1 + (c \partial_x - \tau^{-1} b(x)) v_2 &= \lambda v_2.
\end{aligned}
\]
Labeling $v := v_1$ and substituting the first equation into the second we immediately arrive at equation \eqref{specprob}, with $(v, v_x) \in H^1 \times H^1$. This shows that $\cK (v_1, v_2)^\top = (v, v_x)^\top \in \ker \cT^\tau(\lambda)$.

Now suppose that $(v, v_x)^\top \in \ker \cT^\tau(\lambda) \subset H^1 \times H^1$. Then clearly $v \in H^2$ and let us define $v_1 := v$, $v_2 := \lambda v - cv_x$. It is then easy to verify that
\[
c \partial_x v_1 + v_2 = \lambda v = \lambda v_1, \qquad \text{and, }
\]
\[
\tau^{-1} (\partial_x^2 + a(x)) v_1 + (c\partial_x - \tau^{-1} b(x)) v_2 = \tau^{-1} (\lambda^2 \tau v - c \lambda \tau v_x) = \lambda (\lambda v - cv_x) = \lambda v_2.
\]
This yields $(v_1, v_2)^\top \in \ker (\cL^\tau - \lambda)$. Thus, for each element $(v, v_x)^\top \in \ker \cT^\tau(\lambda)$ there exists $(v_1, v_2)^\top \in \ker (\cL^\tau - \lambda)$ such that $(v, v_x)^\top = \cK (v_2, v_2)^\top$, and we verify that $\cK$ is onto.

Finally, suppose that $\cK (u_1, u_2)^\top = \cK(v_1, v_2)^\top$ for $(u_1, u_2)$, $(v_1, v_2) \in \ker (\cL^\tau - \lambda)$. This means that $(u_1, \partial_x u_1) = (v_1, \partial_x v_1)$ a.e. in $H^2 \times H^1$. But this implies that $v_2 = \lambda v_1 - c\partial_x v_1 = \lambda u_1 - c \partial_x u_1 = u_2$ a.e. in $H^1$ and we conclude that the mapping $\cK$ is one-to one. 
\qed
\end{proof}

An immediate consequence of the one-to-one correspondence between the kernels of $\cL^\tau - \lambda$ and $\cT^\tau(\lambda)$ is that the Fredholm properties of both operators are the same (see, e.g., Sandstede \cite{San02}, section 3.3). Therefore, if we naturally adopt Weyl's definition of spectra and define
\[
\begin{aligned}
\ptsp(\cL^\tau) &:= \{ \lambda \in \C\,: \; \cL^\tau - \lambda \,\text{ is Fredholm with index zero and has a} \\
& \qquad \qquad \qquad \text{non-trivial kernel} \},\\
\ess(\cL^\tau) &:= \{ \lambda \in \C\,: \; \cL^\tau - \lambda \,\text{ is either not Fredholm or has index 
different } \\
& \qquad \qquad \qquad \text{from zero} \},
\end{aligned}
\]
with $\rho(\cL^\tau) = \C \backslash (\ptsp(\cL^\tau) \cup \ess(\cL^\tau))$, then we obtain the following
\begin{corollary}
\label{corsamespec}
For each $\tau > 0$,
\[
\ptsp = \ptsp(\cL^\tau), \quad  \ess = \ess(\cL^\tau), \quad \rho = \rho(\cL^\tau),
\]
where the sets on the left hand sides of the above equalities are, of course, the sets of Definition \ref{defsigmatwo}.
\end{corollary}

For $\lambda$ in the point spectrum, it is clear from Lemma \ref{lemQinv} that the dimensions of the finite-dimensional kernels are the same and, hence, the geometric multiplicity of $\lambda$ remains the same. Moreover, the mapping $\cK$ can also be used to show that the Jordan block structures of $\cL^\tau - \lambda$ and $\cT^\tau(\lambda)$ coincide, that is, the algebraic multiplicity (the length of each maximal Jordan chain) is the same whether computed for one operator or for the other.

\begin{proposition}
\label{propJordan}
The mapping $\cK$ induces a one-to-one correspondence between Jordan chains.
\end{proposition}
\begin{proof}
\smartqed
Suppose $(\varphi, \psi)^\top \in \ker (\cL^\tau - \lambda)$. This implies the following system of equations,
\[
\begin{aligned}
c \varphi_x + \psi &= \lambda \varphi,\\
\tau^{-1} (\partial_x^2 + a(x)) \varphi + (c \partial_x - \tau^{-1} b(x)) \psi &= \lambda \psi.
\end{aligned}
\]
Take the next element in a Jordan chain, say, $(v_1, v_2)^\top \in H^2 \times H^1$ such that
\[
(\cL^\tau - \lambda) \begin{pmatrix} v_1 \\ v_2 \end{pmatrix} = \begin{pmatrix} \varphi \\ \psi \end{pmatrix}.
\]
This yields
\[
\begin{aligned}
c \partial_x v_1 + v_2 - \lambda v_1 &= \varphi,\\
\tau^{-1} (\partial_x^2 + a(x)) v_1 + (c\partial_x - \tau^{-1} b(x)) v_2 - \lambda v_2 &= \psi.
\end{aligned}
\]
Notice that $\cK (v_1, v_2)^\top = (v_1, \partial_x v_1)^\top$, $\cK(\varphi, \psi)^\top = (\varphi, \varphi_x)^\top$. Now substitute $\psi = \lambda \varphi - c \varphi_x$ and $v_2 = \varphi + \lambda v_1 - c\partial_x v_1$ in order to obtain a scalar equation for $v_1$ and $\varphi$. The result is
\[
\tau^{-1} (\partial_x^2 + a(x)) v_1 + (c\partial_x - \tau^{-1} b(x) - \lambda) (\varphi + \lambda v_1 - c\partial_x v_1)  = \lambda \varphi - c \varphi_x.
\]
Labeling $v := v_1$, last equation reads
\[
(1-c^2 \tau) v_{xx} + (cb(x) + 2 \tau \lambda) v_x  - (\lambda^2 \tau v + \lambda b(x) - a(x))v = (b(x) + 2 \tau \lambda) \varphi - 2c\tau \varphi_x,
\]
which is equivalent to
\[
(\partial_x  - \A^\tau(x,\lambda)) \begin{pmatrix} v \\ v_x \end{pmatrix} = \Big( \A_1^\tau(x) + 2 \lambda \A^\tau_2(x) \Big) \begin{pmatrix} \varphi \\ \varphi_x \end{pmatrix}.
\]
Generalizing this procedure, we observe that solutions to
\[
(\cL^\tau - \lambda) \begin{pmatrix} v_1^j \\ v_2^j \end{pmatrix} = \begin{pmatrix} v_1^{j-1} \\ v_2^{j-1} \end{pmatrix},
\]
for some $j \geq 1$, are in one-to-one correspondence to solutions to
\[
\cT^\tau(\lambda) \cK \begin{pmatrix} v_1^j \\ v_2^j \end{pmatrix} = (\partial_\lambda \A^\tau(x,\lambda)) \cK \begin{pmatrix} v_1^{j-1} \\ v_2^{j-1} \end{pmatrix}.
\]
We conclude that a Jordan chain for the operator $\cL^\tau - \lambda$ induces a Jordan chain for $\cT^\tau(\lambda)$ with the same block structure and length.
\qed
\end{proof}

\begin{corollary}
\label{corsameptsp}
Assume $\tau > 0$. Then for any complex number $ \lambda \in \C$ there holds 
\[
\lambda \in \ptsp \quad \text{if and only if} \quad \lambda \in \ptsp(\cL^\tau),
\]
with the same algebraic and geometric multiplicities (here $\ptsp$ is the set in Definition \ref{defsigmatwo}).
\end{corollary}

\begin{remark}
The results of Corollary \ref{corsamespec} and Proposition \ref{propJordan} generalize the spectral equivalence proved in the relaxed Allen-Cahn case (see Section 3 in \cite{LMPS16}). It is remarkable, however, that for the Allen-Cahn model with relaxation the associated matrix $\cL^\tau$ is a first order differential operator, whereas in the present (general) case the operator is of second order. 
\end{remark}

\section{Asymptotic limits and the essential spectrum}
\label{secess}

In this section we analyze the asymptotic equations
\begin{equation}
\label{Wasymp}
W_x = \A^\tau_\pm(\lambda) W, 
\end{equation}
wherupon the asymptotic coefficients are defined in \eqref{asympcoeff}, and 
which will allow us, in turn, to 
locate the essential spectrum of our problem.

\subsection{The asymptotic equations}

Take a look at the asymptotic coefficients \eqref{asympcoeff}. Let us denote the characteristic polynomial of 
$\A^\tau_\pm(\lambda)$ as
\begin{equation}
\label{defcharpol}
p_\pm^\tau(\mu) = \det \big( \A_\pm^\tau (\lambda) - \mu I\big). 
\end{equation}
Notice that $\mu$ is a root of $p_\pm^\tau(\mu) = 0$ if and only if $\kappa = (1-c^2 \tau) \mu$ is a root of
\[
\begin{aligned}
\det \big(\kappa I - (1-c^2 \tau) \A_\pm^\tau (\lambda)\big)  &= \det \begin{pmatrix}
                                                                       \kappa & &  -(1-c^2\tau) \\ - \tau 
\lambda^2 - \lambda b_\pm - |a_\pm| & & \, \kappa + c(b_\pm + 2 \tau \lambda) 
                                                                      \end{pmatrix}\\
&= \kappa^2 + \kappa c(b_\pm + 2 \tau \lambda) - (1 - c^2 \tau)(\tau \lambda^2 + \lambda b_\pm + |a_\pm|) \\ 
&= 0.
\end{aligned}
\]
Suppose that $\kappa = i \xi$, with $\xi \in \R$. Then the $\lambda$-roots of the equation
\begin{equation}
\label{disprel}
\xi^2 - ic\xi (b_\pm + 2 \tau \lambda) + (1-c^2\tau)(\tau \lambda^2 + b_\pm \lambda + |a_\pm|) = 0, 
\end{equation}
define algebraic curves in the complex plane, bounding the essential spectrum. We denote these curves as
\begin{equation}
\label{algcurves}
\lambda = \lambda_{1,2}^\pm(\xi), \qquad \xi \in \R. 
\end{equation}

Equation \eqref{disprel} is the \textit{dispersion relation} for the wave solutions to the constant 
coefficient asymptotic equations.

\begin{remark}
 It is clear that $\lambda = 0$ does not belong to any of the algebraic curves \eqref{algcurves}, inasmuch as 
$\xi^2 - ic\xi b_\pm + (1-c^2\tau) |a_\pm|$ has strictly positive real part for all $\xi \in \R$.
\end{remark}

\subsubsection{The case $\tau = 0$}
We first analyze these curves in the case when $\tau = 0$. Then the dispersion relation \eqref{disprel} reads
\[
 \xi^2 - ic\xi b_\pm + b_\pm \lambda + |a_\pm| = 0,
\]
and the single root is simply
\begin{equation}
\label{algcurvtau0}
\lambda_0^\pm(\xi) = -b_\pm^{-1}|a_\pm| + ic\xi - b_\pm \xi^2,
\end{equation}
for all $\xi \in \R$. These curves lie on the stable half plane with $\Re \lambda < 0$. In fact, there exist
\begin{equation}
\chi_0^\pm = \tfrac{1}{2}b_\pm^{-1}|a_\pm|, 
\end{equation}
\[
 \chi_0 = \min \{\chi_0^+, \chi_0^-\} > 0
\]
such that
\[
 \Re \lambda_0^\pm(\xi) < - \chi_0^\pm \leq -\chi_0 < 0, \qquad \text{for all} \; \xi \in \R. 
\]
In other words, there is a \textit{spectral gap}.

\subsubsection{The case $\tau > 0$}

We now examine the case when $\tau > 0$. Recall that $0 < \tau < 1/c^2$ thanks to the subcharacteristic 
condition. Let us suppose that $\lambda(\xi)$ belongs to one of the curves \eqref{algcurves} and let 
$\eta(\xi) = \Re \lambda(\xi)$, $\beta(\xi) = \Im \lambda(\xi)$. Then, take the real and imaginary parts of 
the dispersion relation \eqref{disprel} to obtain
\begin{equation}
 \label{realp}
\xi^2 + 2c\tau \xi \beta + (1-c^2\tau)\big( \tau(\eta^2 - \beta^2) + \eta b_\pm + |a_\pm|\big) = 0.
\end{equation}
\begin{equation}
\label{imagp}
-c\xi b_\pm - 2c\tau \xi \eta + (1-c^2\tau)\big( 2\tau \eta\beta + b_\pm \beta\big) = 0. 
\end{equation}
 
\begin{remark}
Upon inspection of \eqref{realp} and \eqref{imagp} we notice that if we assume that $\eta = \Re \lambda = 0$ 
for some $\xi \in \R$ then $-c\xi b_\pm + (1-c^2\tau) \beta b_\pm = 0$. Since $b_\pm > 0$ this implies that 
$\beta = c\xi/(1-c^2\tau)$. Substituting into \eqref{realp} we obtain $\xi^2 + \tau c^2 \xi^2 /(1-c^2\tau) + 
|a_\pm| = 0$, which is a contradiction with $|a_\pm| > 0$, $\tau > 0$, $1-c^2\tau > 0$. This shows that the 
algebraic curves never cross the imaginary axis; they remain in either the stable or the unstable complex half 
plane.
\end{remark}

Notice that equation \eqref{imagp} can be written as
\[
 \big( \beta - \frac{c\xi}{1-c^2\tau}\big)(b_\pm + 2\tau \eta) = 0.
\]

Thus, either
\begin{align}
\eta(\xi) &= - \frac{b_\pm}{2\tau}, \label{casei}\\
\textrm{or, } \;\; \beta(\xi) &= \frac{c\xi}{1-c^2\tau}. \label{caseii}
\end{align}

First, let us consider case \eqref{casei}. Substituting into \eqref{realp} yields
\begin{equation}
 \label{eqforbeta}
\tau \beta^2 - \big( \frac{2c\tau \xi}{1-c^2\tau}\big) \beta - |a_\pm| + \frac{b_\pm^2}{4\tau} - 
\frac{\xi^2}{1-c^2\tau} = 0.
\end{equation}
This equation has real solutions $\beta$ provided that
\[
\Delta_1 := \frac{4c^2 \tau^2 \xi^2}{(1-c^2 \tau)^2} - 4\tau \Big( -|a_\pm| + \frac{b_\pm^2}{4\tau} - 
\frac{\xi^2}{1-c^2\tau}\Big) \geq 0, 
\]
or equivalently,
\begin{equation}
 \label{star}
\xi^2 (1-c^2 \tau)^{-2} + |a_\pm| \geq \frac{b_\pm^2}{4 \tau}.
\end{equation}

On the other hand, if we consider case \eqref{caseii} then after substituting into \eqref{realp} we obtain
\begin{equation}
 \label{eqforeta}
\tau \eta^2 + b_\pm \eta + |a_\pm| + \frac{\xi^2}{(1-c^2\tau)^2} = 0.
\end{equation}
Last equation has real solutions $\eta$ if and only if
\[
 \Delta_2 := b_\pm^2 - 4\tau \Big( |a_\pm| + \frac{\xi^2}{(1-c^2\tau)^2} \Big) \geq 0,
\]
that is, when
\begin{equation}
 \label{dstar}
\xi^2(1-c^2\tau)^{-2} + |a_\pm| \leq \frac{b_\pm^2}{4\tau}.
\end{equation}
Therefore, clearly, $\sgn \Delta_2 = - \sgn \Delta_1$. We consider two cases:\\

\noindent \textit{Case (I):} Suppose that for a certain parameter value $\tau > 0$ there holds
\begin{equation}
 \label{taularge}
\frac{b_\pm^2}{4\tau} < |a_\pm|,
\end{equation}
which means that for both the asymptotic states, or for one of them, $\tau > 0$ is sufficiently large such 
that \eqref{taularge} is true. 
\begin{remark}
It is to be observed that this case happens in the example when $g \equiv 1$, $f(u) = u(1-u)(u-1/2)$ if we 
take $\tau = 1$, yielding $b_\pm = 1$, $|a_\pm| = 1/2$. 
\end{remark}

Whence, if \eqref{taularge} holds then condition \eqref{dstar} is never satisfied and \eqref{star} is always 
true. Therefore there are only real solutions for $\beta$ in \eqref{eqforbeta} inasmuch as $\Delta_1 > 0$ for 
all $\xi \in \R$. This implies that the only algebraic curve solutions $\lambda = \lambda(\xi)$ to 
\eqref{disprel} are
\begin{equation}
 \label{etabeta}
\begin{aligned}
\Re \lambda (\xi) = \eta (\xi) &= - \frac{b_\pm}{2\tau}, \\
\Im \lambda (\xi) = \beta(\xi) &= \frac{c\xi}{1-c^2\tau} \pm \frac{1}{2\tau} \sqrt{\Delta_1(\xi)},
\end{aligned}
\end{equation}
for all $\xi \in \R$. Notice that there exists $\chi_1^\pm(\tau) := b_\pm/(4\tau) > 0$ such that there is a 
spectral gap:
\[
 \Re \lambda(\xi) < - \chi_1^\pm < 0, \qquad \xi \in \R.
\]

\noindent \textit{Case (II):} Now suppose that for certain parameter values
\begin{equation}
 \label{tausmall}
\frac{b_\pm^2}{4\tau} \geq |a_\pm|.
\end{equation}

\begin{remark}
Notably, this case occurs for systems of Cattaneo-Maxwell type with $f(u) = u(1-u)(u-\alpha)$, $g(u,\tau) = 1 
- \tau f'(u)$, $\alpha \in (0,1)$. Here $g(u,\tau) > 0$ provided that
\[
 0 < \tau < \frac{3}{1-\alpha + \alpha^2},
\]
as the reader may easily verify. (This warrants hypothesis \eqref{H1} to hold.) Since $g(0,\tau) = b_- = 1 + 
\tau \alpha > 0$, $g(1,\tau) = b_+ = 1+\tau(1-\alpha) > 0$, then clearly
\begin{align*}
\frac{b_-^2}{4\tau} &= \frac{(1+\alpha\tau)^2}{4\tau} \geq \alpha = |a_-|,\\
\frac{b_+^2}{4\tau} &= \frac{(1+(1-\alpha)\tau)^2}{4\tau} \geq 1-\alpha = |a_+|,
\end{align*}
verifying the two conditions \eqref{tausmall}.  
\end{remark}

Assuming \eqref{tausmall}, let $\xi_0^\pm \geq 0$ be the nonnegative solution to
\[
(\xi_0^\pm)^2 = (1-c^2\tau)^2\big( \frac{b_\pm^2}{4\tau} - |a_\pm|\big).
\]
Henceforth, for every $\xi \in (-\xi_0^\pm,\xi_0^\pm)$ we have that
\[
 \xi^2 < (1-c^2\tau)^2\big( \frac{b_\pm^2}{4\tau} - |a_\pm|\big),
\]
condition \eqref{dstar} is satisfied, and consequently, $\Delta_2(\xi) > 0$. In that range for $\xi$ the 
solutions for $\beta$ and $\eta$ are thus given by
\[
 \beta(\xi) = \frac{c\xi}{1-c^2\tau}, \qquad \xi \in (-\xi_0^\pm,\xi_0^\pm),
\]
and by
\begin{equation}
 \label{etagood}
\eta(\xi) = \frac{1}{2\tau} \big( b_\pm \pm \sqrt{\Delta_2(\xi)}\big), \qquad \xi \in (-\xi_0^\pm,\xi_0^\pm),
\end{equation}
respectively. Observe, however, that $\Delta_1(\xi), \Delta_2(\xi) \to 0$ as $|\xi| \uparrow \xi_0^\pm$; that 
$\beta(\xi) \to \pm c\xi_0/(1-c^2\tau)$ as $\xi \to \pm \xi_0^\pm$, $|\xi| < \xi_0^\pm$; and that $\eta(\xi) 
\to -b_\pm/2\tau$ as $|\xi| \uparrow \xi_0^\pm$. This behavior guarantees the continuity of the algebraic 
curves at $|\xi| = \xi_0^\pm$, because the roots of equation \eqref{eqforbeta} at $|\xi|=\xi_0^\pm$ are
\[
 \beta(\xi_0) = \frac{\pm c\xi_0}{1-c^2\tau}
\]
(as $\Delta_1(\xi_0^\pm) = 0$), and $\eta$ is constant, given by $\eta = -b_\pm/2\tau$. Therefore, for values 
$|\xi| \geq \xi_0^\pm$, $\Delta_1$ and $\Delta_2$ switch signs, $\Delta_1$ is now positive and the solutions 
for $\eta$ and $\beta$ are given by formulas \eqref{etabeta}.

Closer inspection of \eqref{etagood} reveals that
\[
 \eta(\xi) = - \frac{b_\pm}{2\tau} \pm \sqrt{\frac{b_\pm^2}{4\tau^2} - \frac{1}{\tau}\Big( |a_\pm| + 
\frac{\xi^2}{(1-c^2\tau)^2}\Big)} \; \leq \, - \frac{b_\pm}{2\tau} + \sqrt{\frac{b_\pm^2}{4\tau^2} - 
\frac{|a_\pm|}{\tau}} \; < 0,
\]
for all $|\xi| \leq \xi_0^\pm$. Therefore, in case (II) there exists
\[
 \chi_2^\pm(\tau) = \frac{b_\pm}{4\tau} - \frac{1}{2} \sqrt{\frac{b_\pm^2}{4\tau^2} - \frac{|a_\pm|}{\tau}} 
>0,
\]
such that
\[
 \Re \lambda(\xi) < - \chi_2^\pm < 0, \qquad |\xi| \leq \xi_0^\pm,
\]
and there is also a spectral gap.

Under these considerations we now define, for each fixed $\tau \geq 0$,
\begin{equation}
 \label{defspectralgap}
0 < \chi_0^\pm (\tau) := \begin{cases}
                         \tfrac{1}{2}b_\pm^{-1}|a_\pm|, & \text{if } \; \tau = 0,\\
                         \displaystyle{\tfrac{1}{2}\Big( \frac{b_\pm}{2\tau} - \sqrt{\frac{b_\pm^2}{4\tau^2} - 
\frac{|a_\pm|}{\tau}} \, \Big)}, & \text{if } \; b_\pm^2 \geq 4\tau |a_\pm|, \, \tau > 0,\\
                         \displaystyle{\frac{b_\pm}{4\tau}}, & \text{otherwise.}
                        \end{cases}
\end{equation}

Thus we have proved the following
\begin{lemma}[spectral gap]
\label{lemspectralgap}
For each $\tau \geq 0$, there exists a uniform
\begin{equation}
 \label{defchi0}
\chi_0(\tau) = \min \{\chi_0^+(\tau), \chi_0^-(\tau)\} > 0,
\end{equation}
(where $\chi_0^\pm (\tau)$ are defined in \eqref{defspectralgap}) such that the algebraic curves $\lambda = 
\lambda_{1,2}^\pm(\xi)$, $\xi \in \R$, solutions to the dispersion relations \eqref{disprel}, satisfy
\begin{equation}
\label{spectralgapeq}
\mathrm{Re}\, \lambda_{1,2}^\pm(\xi) < - \chi_0(\tau) < 0, \qquad \xi \in \R. 
\end{equation}
\end{lemma}

\begin{remark}
The significance of Lemma \ref{lemspectralgap} is that there is no accumulation of essential spectrum at the 
eigenvalue $\lambda = 0$, which is an isolated eigenvalue with finite multiplicity (see Lemma \ref{lemalgm} 
below). 
Notice that for each 
finite $\tau \geq 0$, the bound $\chi_0(\tau)$ is positive. There could be accumulation of the essential 
spectrum in the case when $\tau \to +\infty$ (for which, it may happen, that $\chi_0(\tau) \to 0$), but that 
case is precluded by our hypothesis \eqref{H2}, with an upper bound $\tau < \tau_m < +\infty$. In the case 
of the relaxation model with Cattaneo-Maxwell transfer law (see equation \eqref{ACrelax}), the parameter 
values are bounded by a characteristic relaxation time associated to the reaction, $\tau_m = 1/\max_{u \in 
[0,1]} |f'(u)|$.
\end{remark}

\subsection{Hyperbolicity and consistent splitting}

For a given $\tau \geq 0$, we define the following open, connected region of the complex plane,
\begin{equation}
 \label{defOmega}
\Omega := \{\lambda \in \C \, : \, \Re \lambda > - \chi_0(\tau)\}.
\end{equation}
It properly contains the unstable complex half plane $\C_+ = \{ \Re \lambda > 0\}$. This is called the 
region of consistent splitting \cite{San02}. Denote 
$S^\tau_\pm(\lambda)$ and $U^\tau_\pm(\lambda)$ as the stable and unstable eigenspaces of 
$\A^\tau_\pm(\lambda)$, respectively.

\begin{lemma}\label{lemconsplit}
Given $\tau \geq 0$, for all $\lambda \in \Omega$ the coefficient matrices $\A^\tau_\pm(\lambda)$ have no 
center eigenspace and, moreover,
\[
 \dim S^\tau_\pm(\lambda) = \dim U^\tau_\pm(\lambda) = 1.
\]
\end{lemma}
\begin{proof}
\smartqed
Take $\lambda \in \Omega$ and suppose $\kappa = i \xi$, with $\xi \in \R$, is an eigenvalue of 
$\A_\pm^\tau(\lambda)$. Then $\lambda$ belongs to one of the algebraic curves \eqref{algcurves}. But 
\eqref{spectralgapeq} yields a contradiction with $\lambda \in \Omega$. Therefore, the matrices 
$\A_\pm^\tau(\lambda)$ have no center eigenspace.

Since $\Omega$ is a connected region of the complex plane, it suffices to compute the dimensions of 
$S_\pm^\tau(\lambda)$ and $U_\pm^\tau(\lambda)$ when $\lambda = \eta \in \R_+$, sufficiently large. $\mu$ is 
a root of $p_\pm^\tau(\mu) = \det (\A_\pm^\tau(\lambda) - \mu) = 0$ if and only if $\kappa = (1-c^2 \tau) 
\mu$ is a solution to
\begin{equation}
\label{eqkappa}
 \kappa^2 + \kappa c(b_\pm + 2 \tau \lambda) - (1 - c^2 \tau)(\tau \lambda^2 + \lambda b_\pm + |a_\pm|) = 0.
\end{equation}
Assuming $\lambda = \eta \in \R_+$, the roots are
\[
\kappa = - \frac{c}{2}(b_\pm + 2 \tau \eta) \pm \frac{1}{2} \sqrt{c^2 (b_\pm + 2\tau \eta)^2 + 
4(1-c^2\tau)(\tau \eta^2 + \eta b_\pm + |a_\pm|)}.                                                     
\]
Clearly, for each $\eta > 0$, one of the roots is positive and the other is negative. This proves the lemma.
\qed \end{proof}

The most important consequence of last lemma is the following
\begin{corollary}[stability of the essential spectrum]\label{corstabess}
For each $\tau \geq 0$, the essential spectrum is contained in the stable half-plane. More precisely,
\[
\ess \subset \{\lambda \in \C \, : \, \mathrm{Re} \, \lambda \leq - \chi_0(\tau) < 0\}.
\]
\end{corollary}
\begin{proof} 
\smartqed
The proof follows standard arguments \cite{KaPro13}. Fix $\lambda \in \Omega$. Since $\A_\pm^\tau(\lambda)$ 
are hyperbolic, by exponential dichotomies 
theory (cf. Coppel \cite{Cop78}, Sandstede 
\cite{San02}) the asymptotic systems $W_x = \A_\pm^\tau(\lambda)W$ have exponential dichotomies in $x \in 
\R_+ = (0,+\infty)$ and in $x \in \R_- = (-\infty,0)$, respectively, 
with Morse indices
\[
 \begin{aligned}
i_+(\lambda) &= \dim U_+^\tau(\lambda) = 1,  \\
i_-(\lambda) &= \dim U_-^\tau(\lambda) = 1.
 \end{aligned}
\]
This implies (cf. Palmer \cite{Pal1,Pal2}, Sandstede \cite{San02}), that the variable coefficient operators 
$\cT^\tau(\lambda)$ are Fredholm as well, with index 
\[
\ind \cT^\tau(\lambda) = i_+(\lambda) - i_-(\lambda) = 
0,
\]
showing that $\Omega \subset \C\backslash \ess$, or equivalently, that $\ess \subset \C\backslash 
\Omega = \{\Re 
\lambda \leq - \chi_0(\tau)\}$, as claimed.
\qed \end{proof}

\begin{corollary}\label{corsplit}
For every $\lambda \in \Omega$, the eigenvalues of the asymptotic coefficients \eqref{asympcoeff} are given 
by
\begin{equation}
 \label{evaluesmu}
 \mu^\pm_{1,2}(\lambda) = - \frac{c}{2(1-c^2 \tau)}(b_\pm + 2 \tau \lambda) + \omega^\pm_{1,2}(\lambda),
\end{equation}
whereupon
\[
 \omega_1^\pm (\lambda) := - \frac{1}{2} \Theta_{\pm}(\lambda)^{1/2}, \qquad \omega_2^\pm (\lambda) :=  
\frac{1}{2} \Theta_{\pm}(\lambda)^{1/2},
\]
and,
\[
 \Theta_{\pm}(\lambda) = (1-c^2\tau)^{-2} \Big( c^2 b_\pm^2 + 4(\tau \lambda^2 + b_\pm \lambda + (1-c^2 
\tau)|a_\pm|) \Big).
\]
Morever, for every $\lambda \in \Omega$,
\[
 \mathrm{Re} \, \mu_1^\pm(\lambda) < 0 < \mathrm{Re} \, \mu_2^\pm(\lambda),
\]
that is, $\mu_1^+(\lambda)$ is the decaying mode at $+\infty$, and $\mu_2^-(\lambda)$ is the decaying mode at 
$-\infty$.
\end{corollary}
\begin{proof} 
\smartqed
 Since $p_\pm^\tau(\mu) = 0$ if and only if $\kappa = (1-c^2 \tau)\mu$ is a root of the characteristic 
equation \eqref{eqkappa}, then it is clear that for each $\lambda \in\Omega$ the eigenvalues of 
$\A_\pm^\tau(\lambda)$ are given by \eqref{evaluesmu}. A little algebra yields the expression for the 
discriminant $\Theta_\pm(\lambda)$, an analytic function of $\lambda$. From the proof of Lemma 
\ref{lemconsplit}, we know that, for $\lambda \in \R$ and $\lambda \gg 1$, the only eigenvalue with 
negative real part is $\mu_1^\pm(\lambda)$. Since $\Omega$ is connected and the eigenvalues are continuous 
(analytic) in $\lambda$, we conclude that $\Re \mu_1^\pm(\lambda) < 0$ for all $\lambda \in \Omega$ 
(otherwise, the hyperbolicity, and consequently the consistent splitting, would be violated). The same argument 
applies to $\mu_2^\pm(\lambda)$ and the conclusion follows.
\qed 
\end{proof}

\section{Point spectral stability}

\label{secptsp}

This section is devoted to showing that the point spectrum is stable. The proof presented here makes use of 
energy estimates and contrasts with the one reported in \cite{LMPS16} for the particular case of the 
Allen-Cahn model with relaxation. The former proof was based on a perturbation argument in the vicinity of 
$\tau = 0$ and a further extension to the whole parameter domain. In contrast, here we perform 
energy estimates in the frequency regime that require to apply a transformation 
on the $H^2$-eigenfunction. Thanks to its decaying behaviour, the transformed eigenfunction also belongs to 
$H^2$ and we are able to perform the energy estimates on the new spectral equation. We close the section by 
showing that the eigenvalue $\lambda = 0$ is simple and by stating the main result of the paper.

\subsection{Decay of solutions to spectral equations}

\begin{lemma}
\label{goodw}
Suppose $v \in H^2$ is a solution to the spectral equation \eqref{specprob} for some $\lambda \in 
\ptsp$ with $\mathrm{Re}\, \lambda \geq 0$ and $\lambda \in \Omega$. If we define
\begin{equation}
\label{transfw}
 w(x) = \exp \left( \frac{c}{2(1-c^2 \tau)} \int_{x_0}^x b(s) \, ds \right) v(x), \qquad x \in \R,
\end{equation}
then $w \in H^2$. Here $x_0 \in \R$ is fixed but arbitrary.
\end{lemma}
\begin{proof} 
\smartqed
 Since $\lambda \in \ptsp$ there exists $W = (v,v_x)^\top \in H^1 \times H^1$ such that $\cT^\tau(\lambda) W 
= 0$. This implies, in turn, that $v \in H^2$ is a solution to the spectral equation \eqref{specprob}. To 
analyze the decaying properties of $v$ (equivalently, of $W$) we invoke the Gap Lemma \cite{GZ98,KS98}, which 
relates the decaying properties of the solutions to the variable coefficient system \eqref{Wsystem} to those 
of the solutions of the constant coefficient systems \eqref{Wasymp}, provided that $\A^\tau(x,\lambda)$ 
approaches $\A_\pm^\tau(\lambda)$ exponentially fast as $x \to \pm \infty$. For the precise statement of the 
Gap Lemma we refer the reader to Lemma A.11 in \cite{Zum04}, or Appendix C in \cite{MZ02}.

Suppose that $c > 0$. Since $b > 0$, it is clear that if $x < x_0$ then $|w| \leq |v|$ and $w$ decays like 
$v$ as $x \to - \infty$ Thus, we need to make precise the decaying behaviour of $v$ as $x \to +\infty$. By 
exponential decay of the profile \eqref{expdec}, it is clear that
\[
 |\A^\tau(x,\lambda) - \A_\pm^\tau(\lambda)| \leq C e^{-\nu |x|},
\]
as $x \to \pm \infty$, for some $C, \nu > 0$, uniformly in $\lambda$. Then, applying the Gap Lemma and 
Corollary \ref{corsplit}, the decaying solution $W$ at $+\infty$ to the variable coefficient equation behaves 
as
\[
 W(x,\lambda) = e^{\mu_1^+(\lambda)} \Big( V_1^+(\lambda) + O(e^{-\nu|x|} |V_1^+(\lambda)|) \Big), \quad x > 
0,
\]
where $V_1^+(\lambda)$ is the eigenvector of $\A_\pm^\tau(\lambda)$ associated to the eigenmode 
$\mu_1^+(\lambda)$. This imples that $v$ and $v_x$ decay, at most, as
\[
 |v|, |v_x| \leq C e^{\Re \mu_1^+(\lambda) x},
\]
as $x \to +\infty$. We then readily see, from Corollary \ref{corsplit}, that
\[
\begin{aligned}
|w| &\leq C \exp \Big( \frac{c}{2(1-c^2 \tau)} \int_{x_0}^x |b(s) - b_+| \, ds \Big) \times \\ & \qquad 
\times \exp \Big( \Big( - \frac{c \tau \Re \lambda}{2(1-c^2 \tau)} - \frac{1}{2 \sqrt{2}} \sqrt{\Re 
\Theta_+(\lambda) + |\Theta_+(\lambda)|}\Big) x \Big) \\
&\leq C \exp \Big( - \frac{c \tau (\Re \lambda)x}{2(1-c^2 \tau)}\Big) \exp \Big( - \frac{x}{2 \sqrt{2}} 
\sqrt{\Re \Theta_+(\lambda) + 
|\Theta_+(\lambda)|} \Big) \, \to 0, 
\end{aligned}
\]
as $x \to +\infty$, thanks to exponential decay of the profile, which yields
\[
 \exp \Big( \frac{c}{2(1-c^2 \tau)} \int_{x_0}^x |b(s) - b_\pm| \, ds \Big) \leq C \exp( -e^{-\nu x}) \leq C.
\]

This shows that $w$ decays exponentially fast as $x \to +\infty$. Since $v_x$ decays as the same rate as $v$, 
it is easy to verify that $w_x$ also decays exponentially fast at $+\infty$. We conclude that $w \in H^1$. 
Upon differentiation one can prove that, in fact, $w \in H^2$, as $w_{xx}$ decays exponentially fast as well 
at $+\infty$. Details are left to the dedicated reader. 

The case $c < 0$ can be treated similarly, inasmuch as the decay at $-\infty$ of the eigenfunction $W = 
(v,v_x)^\top$ is determined by the eigenmode $\mu_2^-(\lambda)$; an analogous argument applies. This 
concludes the proof of the lemma.
\qed \end{proof}

\subsection{Energy estimates}

Suppose that $\lambda \in \ptsp$, with $\Re \lambda \geq 0$ (and consequently, $\lambda \in \Omega$). Then 
there exists $W = (v, v_x)^\top \in H^1 \times H^1$ such that 
$\cT^\tau (\lambda) W = 0$. This is tantamount to have an $H^2$ solution $v$ to the spectral equation 
\eqref{specprob}. Consider the transformation
\[
 v(x) = w(x) e^{\theta(x)},
\]
where the function $\theta = \theta(x)$ is to be determined. Upon substitution into 
\eqref{specprob} we obtain 
\[
\begin{aligned}
 \lambda^2 \tau w - 2c\lambda \tau (w_x + \theta_x w) + \lambda b(x) w &= (1 - c^2 \tau) w_{xx} + \big( 
2(1-c^2 \tau) \theta_x + c b(x) \big) w_x + \\
& \; + \big( (1-c^2 \tau) (\theta_x^2 + \theta_{xx}) + cb(x) \theta_x + a(x) \big) w.
\end{aligned}
\]
Choose $\theta$ such that
\[
 \theta_x = - \frac{c}{2(1-c^2\tau)} b(x).
\]
This yields
\begin{equation}
 \label{eqsix}
 \lambda^2 \tau w - 2c \lambda \tau w_x + \frac{\lambda b(x)}{1-c^2 \tau} w =  (1 - c^2 \tau) w_{xx} + H(x) w,
\end{equation}
whereupon
\[
 H(x) := a(x) - \frac{c^2 b(x)^2}{4(1-c^2 \tau)} - \tfrac{1}{2} c b'(x).
\]

If we apply the same procedure to the eigenfunction $U_x \in H^2$ associated to the eigenvalue $\lambda = 0 
\in \ptsp$, denoting $U_x = \psi e^\theta$ we arrive at
\begin{equation}
 \label{eqsixpsi}
 0 =  (1 - c^2 \tau) \psi_{xx} + H(x) \psi.
\end{equation}
By monotonicity of the profile, $U_x > 0$, we know that $\psi > 0$ and we can solve for $H$ in 
\eqref{eqsixpsi}, yielding
\[
 H(x) = - (1-c^2 \tau) \frac{\psi_{xx}}{\psi}.
\]
Substituting back into \eqref{eqsix} we obtain
\begin{equation}
 \label{eqsix2}
 \lambda^2 \tau w - 2c \lambda \tau w_x + \frac{\lambda b(x)}{1-c^2 \tau} w =  (1 - c^2 \tau) \Big( w_{xx} - 
\frac{\psi_{xx}}{\psi} w \Big).
\end{equation}
Notice that thanks to Lemma \ref{goodw}, we have that this is an spectral equation for $w \in H^2$. We 
perform standard energy estimates on equation \eqref{eqsix2}. Multiply by $\overline{w}$ and integrate by 
parts in $\R$. The result is,
\[
\begin{aligned}
 \lambda \tau^2 \| w \|_{L^2}^2 - 2c \lambda \tau \int_{\R} \overline{w} w_x \, dx &+ \frac{\lambda}{1-c^2 
\tau} \int_{\R} b(x) |w|^2 \, dx = \\ &= (1-c^2 \tau) \left( - \int_{\R} |w_x|^2 \, dx + \int_{\R} \psi_x 
\partial_x \Big( \frac{|w|^2}{\psi} \Big) \, dx \right)
\end{aligned}
\]
Using the identity
\begin{equation*}
\psi^{2}\left|\left(\frac{w}{\psi} \right)_x \right|^{2} = - \left( \psi_x \left(\frac{|w|^{2}}{\psi} 
\right)_x-|w_x|^{2} \right),
\end{equation*}
and substituting, we obtain the estimate
\begin{equation}
\label{basicee}
\begin{aligned}
 \lambda \tau^2 \| w \|_{L^2}^2 - 2c \lambda \tau \int_{\R} \overline{w} w_x \, dx &+ \frac{\lambda}{1-c^2 
\tau} \int_{\R} b(x) |w|^2 \, dx = \\ &= - (1-c^2 \tau) \int_{\R} \psi^2 \left| \partial_x \Big( 
\frac{w}{\psi}\Big) \right|^2 \, dx.
\end{aligned}
\end{equation}

\begin{lemma}[point spectral stability]\label{lemptsp}
Suppose $\tau \geq 0$. If $\lambda \in \ptsp \cap \Omega$ then either $\lambda = 0$, or $\mathrm{Re} \, 
\lambda \leq -
\chi_1(\tau) < 0$, for some uniform $\chi_1(\tau) > 0$.
\end{lemma}
\begin{proof} 
\smartqed
The result is a consequence of the basic energy estimate \eqref{basicee}. Indeed, suppose that $\lambda \in 
\ptsp$ and $\Re \lambda \geq 0$ (and consequently, $\lambda \in \Omega$). Then after the transformation, $w = 
e^{-\theta} v \in H^2$ satisfies \eqref{basicee}. Notice that
\[
 \Re \int_{\R} \overline{w} w_x \, dx = \tfrac{1}{2} \int_{\R} \partial_x \big( |w|^2\big) \, dx = 0.
\]
First, let us assume that $\tau > 0$. For shortness, we denote,
\[
\begin{aligned}
 k_0 &:= (1-c^2 \tau) \int_{\R} \psi^2 \left| \partial_x \Big( \frac{w}{\psi}\Big) \right|^2 \, dx &\geq 0,\\
 k_1 &:= (1-c^2 \tau)^{-1} \int_{\R} b(x) |w|^2 \, dx &> 0,\\
 k_2 &:= \tau^2 \| w \|_{L^2}^2 &> 0,\\
 i k_3 &:= \int_{\R} \overline{w} w_x \, dx,
\end{aligned}
\]
with $k_j \in \R$. Notice that $k_1, k_2 > 0$ because $v$ is an eigenfunction, $\tau > 0$, and because 
of \eqref{H2}.

Let us denote $\zeta = \Re \lambda$, $\beta = \Im \lambda$. Therefore, taking the real and imaginary parts of 
\eqref{basicee} yields
\[
 \begin{aligned}
  (\zeta^2 - \beta^2) k_2 + 2c\tau \beta k_3 + \zeta k_1 + k_0 &= 0,\\
  2 \zeta \beta k_2 - 2c\tau \zeta k_3 + \beta k_1 &= 0.
 \end{aligned}
\]
Multiply the first equation by $\zeta$, the second by $\beta$, and add them up. The result is,
\[
 (\zeta^2 + \beta^2) (k_1 + \zeta k_2) + \zeta k_0 = 0,
\]
or, equivalently,
\[
 |\lambda|^2 k_1 + (\Re \lambda) \big( k_0 + |\lambda|^2 k_2 \big) = 0.
\]
Since $k_0 > 0$, $k_1, k_2 \geq 0$, this implies that $\Re \lambda \leq 0$.

Now, if we assume that $\zeta = \Re \lambda = 0$, from the equations we have that $\beta^2 k_1 = 0$. Since 
$k_1 > 0$ we conclude that $\beta = 0$ and this implies that $\lambda = 0$. On the other hand, if we assume 
that $\beta = \Im \lambda = 0$, then from the first equation we obtain,
\[
 k_2 \zeta^2 + k_1 \zeta + k_0 = 0, 
\]
or,
\[
 \zeta = \Re \lambda = - \frac{k_1}{2 k_2} \pm \frac{1}{2 k_2} \Big( k_1^2 - 4 k_2 k_0 \Big)^{1/2}.
\]
Since $k_2 k_0 \geq 0$ we have that $\Re \lambda = \zeta < 0$, a contradiction. 

We conclude that the only 
eigenvalue with $\Re \lambda = 0$ 
is $\lambda = 0$ and that, for any other eigenvalue with $\lambda \neq 0$ in $\Omega$, there holds
\[
 \Re \lambda \leq - \chi_1(\tau) < 0,
\]
for some $\chi_1(\tau) > 0$. This holds because the set $\ptsp$ comprises isolated eigenvalues with finite multiplicity. 
$-\chi_1(\tau) < 0$ is actually the real part of the first (isolated) eigenvalue different from zero. In 
other words, there is a spectral gap. 

In the case where $\tau = 0$, the basic energy estimate \eqref{basicee} yields
\[
 \lambda \int_{\R} b(x) |w|^2 \, dx = - \int_{\R} \psi^2 \left| \partial_x \Big( 
\frac{w}{\psi}\Big) \right|^2 \, dx,
\]
which implies, in turn, that $\lambda \in \R$ and $\lambda \leq 0$. 

Finally, notice that $\lambda = 0$ if and only if 
$w/\psi = 0$ a.e., which is tantamount to $v = U_x$ a.e. This concludes the proof of the lemma.
\qed \end{proof}

As a consequence of the proof of Lemma \ref{lemptsp} we have the following immediate
\begin{corollary}
\label{corgm}
 $\lambda = 0$ is an eigenvalue with $g.m. = 1$.
\end{corollary}

\subsection{Simple translation eigenvalue}

We now show that the eigenvalue $\lambda = 0$ is a simple eigenvalue.

\begin{lemma}\label{lemalgm}
 The algebraic multiplicity of $\lambda = 0 \in \ptsp$ is equal to one.
\end{lemma}
\begin{proof}
\smartqed
From Corollary \ref{corgm}, we know that $\Phi = (U_x,U_{xx})^\top \in H^1 \times H^1$ is the only 
eigenfunction associated to $\lambda = 0$. Let us denote, for simplicity, $\phi = U_x \in H^2$, 
so that $\Phi = (\phi, \phi_x)^\top$. Clearly, because of equation \eqref{nprofileq}, $\phi \in H^2$ is a 
solution to 
\[
 \cA \phi := (1-c^2 \tau) \phi_{xx} + cb(x) \phi_x + a(x) \phi = 0.
\]
This holds upon differentiation \eqref{nprofileq} with respect to $x$. The auxiliary operator, $\cA : L^2 \to 
L^2$ defined above, with domain $\cD(\cA) = H^2$, has a formal adjoint, $\cA^* : L^2 \to L^2$, given by
\[
 \cA^* \psi = (1-c^2 \tau) \psi_{xx} - cb(x) \psi_x + (a(x) -cb'(x)) \psi, \qquad \psi \in \cD(\cA^*) = H^2 \subset 
L^2.
\]

Now, for any $\lambda \in \ptsp$, the operator $\cT^\tau(\lambda)$ is Fredholm with index zero. Therefore, by 
properties of closed operators \cite{Kat80}, we have that
\[
 \dim \ker \cT^\tau(\lambda)^* = \dim \cR(\cT^\tau(\lambda))^\perp = \mathrm{codim} \, \cR(\cT^\tau(\lambda)) 
= \dim \ker \cT^\tau(\lambda).
\]
Since $\dim \ker \cT^\tau(0) = 1$ we conclude that there exists a unique bounded solution $\Psi = 
(y,z)^\top \in H^1 \times H^1$ to the adjoint equation
\[
 \cT^\tau(0)^* \Psi = - \big( \partial_x + \A^\tau(x,0)^*\big)\Psi = 0.
\]

From the expression for $\A^\tau(x,0)$ we observe that $(y,z)^\top \in H^1 \times H^1$ is a solution to the 
system
\begin{equation}
 \label{yzsyst}
 \begin{aligned}
 -a(x)z + (1-c^2 \tau) y_x &= 0,\\
 (1-c^2 \tau) y - cb(x) z + (1-c^2\tau) z_x &= 0.
 \end{aligned}
 \end{equation}
Since the coefficents are bounded and $y, z \in H^1$, by a bootstrapping argument we can verify from the 
system of equations that actually $y, z \in H^2$. Thus, upon differentiation of the second equation and 
substitution into the first one we obtain
\[
 \cA^* z = (1-c^2 \tau) z_{xx} - cb(x) z + (a(x) - b'(x)) z = 0.
\]
We conclude that $z = z(x)$ is the only bounded $H^2$-solution to $\cA^* z = 0$.

Now, like in \cite{LMPS16}, let us define the Melnikov integral
\[
 \Gamma := \langle \Psi, \big( \partial_\lambda \A^\tau(x,\lambda) \big)_{|\lambda = 0} \Phi \rangle_{L^2 
\times L^2}.
\]
It is well-known (see section 4.2.1 in \cite{San02}) that $\Gamma$ decides whether $\lambda = 0$ is a simple 
eigenvalue: if $\Gamma \neq 0$ then its algebraic multiplicity is equal to one (see also \cite{LMPS16} and the 
discussion therein). From \eqref{derivlamA} we observe that $\partial_\lambda \A^\tau(x,\lambda) 
\big)_{|\lambda = 0} = \A_1^\tau(x)$, and therefore we arrive at
\[
 \begin{aligned}
 \Gamma = \langle \Psi, \A_1^\tau(x) \Phi \rangle_{L^2 \times L^2} &= \int_{\R} \begin{pmatrix}
               y \\ z
              \end{pmatrix}^* \A_1^\tau(x) \begin{pmatrix}
               \phi \\ \phi_x
              \end{pmatrix} \, dx \\
              &= (1-c^2 \tau)^{-1} \int_{\R} \overline{z} \big( b(x) \phi - 2c\tau \phi_x \big) \, dx.
\end{aligned}
\]

Like in the argumentation leading to the proof of Lemma 3.2 in \cite{LMPS16}, a direct computation allows to 
verify that the only bounded solution to $\cA^* z = 0$ is given by $z =\phi/h^2$, where $h$ is a solution to
\[
 h_x = - \frac{c b(x)}{2 (1- c^2 \tau)} h,
\]
that is, $h(x) = e^{\theta(x)}$ as in the previous section. By the arguments of Lemma \ref{goodw} it is easy 
to verify that $z \in H^2$ inasmuch as $\phi \in H^2$. Thus, a direct computation yields
\[
 \cA^* z = \frac{1}{h^2}\Big( (1-c^2 \tau) \phi_{xx} + cb(x) \phi_x + a(x) \phi \Big) = \frac{1}{h^2} \cA 
\phi = 0,
\]
as claimed. Whence, substituting $z = \phi / h^2$ into the expression for $\Gamma$ we obtain
\[
 \begin{aligned}
(1-c^2 \tau) \Gamma &= \int_{\R} \frac{\overline{\phi}}{h^2} \big( b(x) \phi - 2c\tau \phi_x \big) \, dx \\
&= \int_{\R} \frac{b(x)}{h^2} |\phi|^2 - \frac{c\tau}{h^2} \partial_x ( |\phi|^2 ) \, dx \\
&= (1 - c^2 \tau)^{-1} \int_{\R} \frac{b(x)}{h^2} |\phi|^2 \, dx,
 \end{aligned}
\]
after integration by parts and substitution of the equation for $h$. We observe that
\[
 \Gamma = (1 - c^2 \tau)^{-2} \int_{\R} \frac{b(x)}{h^2} |\phi|^2 \, dx > 0,
\]
and the conclusion follows.
\qed
\end{proof}

\subsection{Main result}

We conclude this section by stating our main theorem.

\begin{theorem}[spectral stability with spectral gap]
\label{mainthm}
Under assumptions \eqref{H1} and \eqref{H2}, for each $\tau \in [0,\tau_m)$ fixed let $U = U(x)$ be the 
monotone traveling front solution to \eqref{hypAC}. Then this front is spectrally stable with spectral gap, 
more precisely, there exists a uniform $\chi(\tau) > 0$ such that
\[
 \sigma \subset \{ \lambda \in \C \, : \, \mathrm{Re} \, \lambda \leq - \chi(\tau) < 0 \} \cup \{ 0 \}.
\]
Moreover, $\lambda = 0$ is a simple isolated eigenvalue (with algebraic multiplicity equal to one) associated 
to translation invariance.
\end{theorem}
\begin{proof}
\smartqed
The conclusion follows directly by collecting the results of Corollary \ref{corstabess} and Lemmata \ref{lemptsp} 
and \ref{lemalgm}. The spectral gap is given by
\[
\chi(\tau) := \min \{ \chi_0(\tau), \chi_1(\tau) \} > 0,
\]
for each fixed $\tau \in [0, \tau_m)$, where $\chi_0$ is defined in \eqref{defchi0} and $\chi_1$ is the gap defined in Lemma \ref{lemptsp}.
\qed
\end{proof}

Notice that if $\tau > 0$ then the statement of Theorem \ref{mainthm} can be recast in terms of the spectrum of the operators $\cL^\tau$ defined in \eqref{defcLtau}. Indeed, corollaries \ref{corsamespec} and \ref{corsameptsp} imply that the spectral stability with spectral gap are also properties of the matrix operators $\cL^\tau$ when $\tau > 0$. Thus, we can state the following

\begin{theorem}
Under assumptions \eqref{H1} and \eqref{H2}, and for each fixed $0 < \tau < \tau_m$ there holds
\[
\sigma(\cL^\tau) \subset \{ \lambda \in \C \, : \, \Re \lambda < - \chi(\tau) < 0\} \cup \{ 0 \},
\]
for some uniform $\chi(\tau) > 0$. Moreover, $\lambda = 0$ is a simple isolated eigenvalue of $\cL^\tau$ with associated eigenfunction $(U_x, -cU_{xx}) \in \ker \cL^\tau$.
\end{theorem}

\section{Discussion}
\label{secdisc}
In this paper we established the spectral stability with spectral gap of a family of traveling fronts for nonlinear wave equations of the form \eqref{hypAC} when the reaction function is of bistable type. The equations under consideration are endowed with a positive ``damping" term, $g > 0$, which generalizes the previous studied case of the Allen-Cahn equation with relaxation. To that end, we revisited the existence theory using a dynamical systems approach, more in the spirit of our previous work \cite{LMPS16}. Even though existence results are available in the literature \cite{GiKe15}, here we presented a different construction which allows us to derive a variational formula for the unique wave speed and to establish exponential decay of the profile function. Both features play a role in the stability analysis: the uniqueness of the speed is related to the algebraic multiplicity of the zero eigenvalue of the linearized problem around the front, whereas the exponential decay is crucial to locate the essential spectrum.

Our main result establishes that the spectrum of the linearized problem around the front is located in the complex half plane with negative real part, except for the translation zero eigenvalue, which is isolated with finite multiplicity. This property is also known as \textit{spectral stability with spectral gap} and prevents the accumulation of essential spectrum around zero. In this fashion, we generalize the analysis performed in \cite{LMPS16} for a particular case (the Allen-Cahn equation with relaxation) to a wider class of equations. It is important to remark that this result is more general not only in applicability but also in methodology. Indeed, the present proof makes use of energy estimates and works for the whole parameter regime, whereas the previous argument is of perturbative nature, with an extension to further relaxation times. In our opinion, the method presented here is more direct.

The establishment of spectral stability is a first step in a more general program which includes the nonlinear stability analysis of the fronts under small perturbations. Thanks to the location of the spectrum in the complex plane, we conjecture that the linearized operator around the wave is the infinitesimal generator of a $C_0$-semigroup. The generation of such semigroup and its decaying properties is a matter of future investigation. (As additional information, in the Appendix we present how to establish resolvent estimates in the case of stationary fronts with $c = 0$, yielding the generation of the semigroup via Lumer-Philips theorem.) Such analysis, also called \textit{linearized stability} in the literature \cite{KaPro13,San02}, is used to prove nonlinear stability in a key way. There exist results in the literature which guarantee nonlinear stability under the assumption of spectral stability (see, e.g., Rottmann-Matthes \cite{Rott11,Rott12a}), but they are not applicable to the generic class of equations considered here, as they are restricted to hyperbolic systems with constant coefficient first order operators. We regard the nonlinear stability of the hyperbolic fronts of equations of the form \eqref{hypAC} as an important open problem which warrants attention from the nonlinear wave propagation community.

\begin{acknowledgement}
R. G. Plaza is grateful to the Department of Information Engineering, Computer Science and Mathematics of the 
University of L'Aquila, for their hospitality during the Fall of 2017, when this research was carried out. 
This work was partially supported by the EU Project ModComShock G.A. N. 
642768.

\end{acknowledgement}

\section*{Appendix: Resolvent estimates for stationary fronts}
\addcontentsline{toc}{section}{Appendix}

Fix $\tau > 0$ and consider the space $\mathcal{X} := H^1 \times L^2$ endowed with the scalar product
\begin{equation*}
	\langle (u_1,v_1),(u_2,v_2)\rangle_{{}_{\mathcal{X}}} 
		:=\Re \langle u_1,u_2\rangle_{{}_{L^2}}
			+\tau^{-1}\Re \langle \partial_x u_1, \partial_x u_2\rangle_{{}_{L^2}}
			+\Re \langle v_1,v_2\rangle_{{}_{L^2}},
\end{equation*}
and corresponding norm
\begin{equation*}
	\|(u,v)\|_{{}_{\mathcal{X}}}= \Big( \|u\|_{{}_{L^2}}^2+\tau^{-1}\|u_x \|_{{}_{L^2}}^2 + \|v\|_{{}_{L^2}}^2 \Big)^{1/2}.
\end{equation*}

Then, for simplicity drop the $\tau > 0$ from the notation and consider the operator defined in \eqref{defcLtau},
\[
\cL \begin{pmatrix} u \\ v \end{pmatrix} = \begin{pmatrix} c \partial_x & & 1 \\ \tau^{-1} (\partial_x^2 + a(x)) & & \; c \partial_x - \tau^{-1} b(x) \end{pmatrix} \begin{pmatrix} u \\ v \end{pmatrix},
\]
as a closed, densely defined operator on $\cX$ with domain $\cD = H^2 \times H^1$. This operator can be conveniently written as
\[
\cL = \cL_0 + \cB,
\]
where
\[
\cL_0 := \begin{pmatrix} c \partial_x & & 1 \\ \tau^{-1} \partial_x^2 -1 & & c \partial_x - \tau^{-1}b(x) \end{pmatrix}, \qquad \cB := \begin{pmatrix} 0 & & 0 \\ \tau^{-1} a(x) + 1 & & 0 \end{pmatrix}.
\]

We first observe that the operator $\cL_0$ is dissipative on $\mathcal{X}$ since for any $\mathbf{w} = (u,v)^\top \in \cD$,
\begin{equation*}
	\begin{aligned}
	\langle \mathbf{w}, \cL_0 \mathbf{w} \rangle_{{}_{\mathcal{X}}} &= \Re \langle u, c u_x + v  \rangle_{{}_{L^2}} 
		+ \tau^{-1} \Re \langle u_x, cu_{xx} + v_x \rangle_{{}_{L^2}}  + \\ & \;\;\; + \Re \langle v, \tau^{-1} u_{xx} -u + cv_x - \tau^{-1} b(x) v \rangle_{{}_{L^2}} \\
	&= - \tau^{-1} \Re \langle v, b(x)v \rangle_{{}_{L^2}} \leq 0,\\
	\end{aligned}
\end{equation*}
in view of Hypothesis \eqref{H2} and having used the fact that $\Re \langle f, f_x\rangle_{{}_{L^2}} = 0$ for any $f \in H^1$. Since $\cD$ is dense in $\cX$ and by dissipativity, thanks to the Lumer-Philips theorem (see, e.g., Theorem 12.22 in \cite{ReRo04}) it suffices to show that $\cL_0 - \lambda$ is onto for real $\lambda$ sufficiently large to conclude that $\cL_0$ is the infinitesimal generator of a $C_0$-semigroup of contractions, $e^{t \cL_0}$, satisfying $\| e^{t \cL_0}\| \leq 1$. Clearly, $\cB$ is a bounded operator and $\|\cB\| \leq O(1 + \tau^{-1} \| a \|_{{}_{L^\infty}})$; since $\cL$ is a bounded perturbation of $\cL_0$, it is also the infinitesimal generator of a quasi-contractive $C_0$-semigroup, $\cS(t)$, such that
\[
\| \cS(t) \| \leq e^{t \| \cB \|} = e^{tC(1 + \tau^{-1} \| a \|_{{}_{L^\infty}})},
\]
for some $C > 0$ (see Theorem 1.1 in Pazy \cite{Pazy83}, chapter 3).

We illustrate how to prove that $\cL_0 - \lambda$ is onto for $\lambda$ real and large in the case of a stationary front with $c = 0$ by establishing a resolvent estimate. 

First, note that if $c = 0$ then the operator $\cL_0$ reduces to
\begin{equation}
\label{defopc0}
\cL_0 \begin{pmatrix} u \\ v \end{pmatrix} = \begin{pmatrix} 0 & & 1 \\ \tau^{-1} \partial_x^2 - 1 & & \; - \tau^{-1} b(x) \end{pmatrix} \begin{pmatrix} u \\ v \end{pmatrix}.
\end{equation}
For given $(\phi, \psi)^\top \in \cX$ suppose that $(u,v)^\top \in \cD$ is a solution to the resolvent equation
\[
(\lambda - \cL_0) \begin{pmatrix} u \\ v \end{pmatrix} = \begin{pmatrix} \phi \\ \psi / \tau  \end{pmatrix},
\]
for some $\lambda \in \C$. This yields the system of equations
\begin{equation}\label{reszerospeed}
	\lambda u-v=\phi,\qquad
	\tau\lambda v-u_{xx} + \tau u+b(x)v=\psi.
\end{equation}

\begin{lemma}\label{lemma:vuprime}
Let $b(x)\geq b_0>0$ for any $x\in\R$.
Given $(\phi,\psi)\in \cX = H^1 \times L^2$, let $(u,v)\in \cD = H^2 \times H^1$
be a solution to system \eqref{reszerospeed}.
Then for any $r>0$, there exists a constant $C>0$ (depending on $\tau, b_0$
and $r$) such that
\begin{equation}\label{vuprime}
	\|v\|_{{}_{L^2}}+\|u_x\|_{{}_{L^2}}
		\leq C\left( \|\psi\|_{{}_{L^2}}+\|\phi_x\|_{{}_{L^2}}+\|u\|_{{}_{L^2}}\right)
\end{equation}
for any $\lambda$ with $|\lambda|\geq r>0$ and $\Re \lambda\geq 0$.
\end{lemma}

\begin{proof}
\smartqed
Multiplying the second equation by $\bar v$ we obtain
\begin{equation*}
	\bigl(\tau\lambda+b(x)\bigr) |v|^2-(u_x \bar v)_x + u_x \bar{v}_x + \tau u \bar{v}=\psi \,\bar{v}.
\end{equation*}
Since $v_x = \lambda u_x - \phi_x$, there holds
\begin{equation*}
	\bigl(\tau\lambda+b(x)\bigr) |v|^2+\bar \lambda |u_x|^2 + \tau u\,\bar v-(u_x \bar{v}_x)_x
		=\psi\,\bar{v}+\bar{\phi}_x \,u_x.
\end{equation*}
Integrating in $\R$ and separating real and imaginary parts, we infer
\begin{equation*}
	\begin{aligned}
		&\bigl(\tau\Re \lambda+b_0\bigr) \|v\|_{{}_{L^2}}^2+\Re \lambda \|u_x\|_{{}_{L^2}}^2
		\leq \|\psi\|_{{}_{L^2}}\|v\|_{{}_{L^2}}+\|\phi_x\|_{{}_{L^2}}\|u_x\|_{{}_{L^2}}
			+\tau \|u\|_{{}_{L^2}}\|v\|_{{}_{L^2}},\\
		&\bigl|\Im\lambda\bigr|\,\Bigl|\tau\|v\|_{{}_{L^2}}^2-\|u_x\|_{{}_{L^2}}^2\Bigr|
		\leq \|\psi\|_{{}_{L^2}}\|v\|_{{}_{L^2}}+\|\phi_x\|_{{}_{L^2}}\|u_x\|_{{}_{L^2}}
			+ \tau \|u\|_{{}_{L^2}}\|v\|_{{}_{L^2}}.\\
	\end{aligned}
\end{equation*}
Applying Young's inequality, we deduce
\begin{equation*}
	\begin{aligned}
		&\bigl(\tau \, \Re\lambda+b_0\bigr) \|v\|_{{}_{L^2}}^2+\Re \lambda \|u_x\|_{{}_{L^2}}^2
		\leq  \frac{1}{b_0}\|\psi\|_{{}_{L^2}}^2+\frac{\tau^2}{b_0}\|u\|_{{}_{L^2}}^2
			+\|\phi_x\|_{{}_{L^2}}\|u_x\|_{{}_{L^2}}+\frac{b_0}{2}\|v\|_{{}_{L^2}}^2.
	\end{aligned}
\end{equation*}
Hence, the following two estimates hold for any choice of $\lambda$ such that
$\Re\lambda\geq 0$,
\begin{equation}\label{forvuprime}
	\begin{aligned}
		&\frac12\,b_0\|v\|_{{}_{L^2}}^2+\Re \lambda \|u_x\|_{{}_{L^2}}^2
		\leq  \frac{1}{b_0}\|\psi\|_{{}_{L^2}}^2+\frac{\tau^2}{b_0}\|u\|_{{}_{L^2}}^2
			+\|\phi_x\|_{{}_{L^2}}\|u_x\|_{{}_{L^2}}\\
		&\bigl|\Im\lambda\bigr|\,\Bigl|\tau\|v\|_{{}_{L^2}}^2-\|u_x\|_{{}_{L^2}}^2\Bigr|
		\leq \|\psi\|_{{}_{L^2}}\|v\|_{{}_{L^2}}+ \tau \|u\|_{{}_{L^2}}\|v\|_{{}_{L^2}}
			+\|\phi_x\|_{{}_{L^2}}\|u_x\|_{{}_{L^2}}
	\end{aligned}
\end{equation}
For $\Re\lambda\geq c_0>0$, there holds
\begin{equation*}
	\begin{aligned}
	b_0 \|v\|_{{}_{L^2}}^2+c_0\|u_x\|_{{}_{L^2}}^2
		&\leq \frac{1}{b_0}\|\psi\|_{{}_{L^2}}^2+\frac{\tau^2}{b_0}\|u\|_{{}_{L^2}}^2
			+\|\phi_x\|_{{}_{L^2}}\|u_x\|_{{}_{L^2}}\\
		&\leq \frac{1}{b_0}\|\psi\|_{{}_{L^2}}^2+\frac{1}{2c_0} \|\phi_x\|_{{}_{L^2}}^2
			+\frac{\tau^2}{b_0}\|u\|_{{}_{L^2}}^2+\frac{c_0}{2}\|u_x\|_{{}_{L^2}}^2.
	\end{aligned}
\end{equation*}
Thus, we deduce
\begin{equation*}
	\|v\|_{{}_{L^2}}^2+\|u_x\|_{{}_{L^2}}^2
		\leq C\left(\|\psi\|_{{}_{L^2}}^2+\|\phi_x\|_{{}_{L^2}}^2+\|u\|_{{}_{L^2}}^2\right),
\end{equation*}
for some strictly positive constant $C$ depending on $b_0, \tau$ and $c_0$.

Next, let $\lambda$ to be such that $|\Im\lambda|\geq \theta_0>0$.
Then, from the second bound in \eqref{forvuprime}, it follows
\begin{equation*}
	\begin{aligned}
	\theta_0\|u_x\|_{{}_{L^2}}^2
		&\leq \|\psi\|_{{}_{L^2}}\|v\|_{{}_{L^2}}+\|\phi_x\|_{{}_{L^2}}\|u_x\|_{{}_{L^2}}
			+\tau\|u\|_{{}_{L^2}}\|v\|_{{}_{L^2}}+\theta_0\tau\|v\|_{{}_{L^2}}^2\\
		&\leq \|\psi\|_{{}_{L^2}}\|v\|_{{}_{L^2}}
			+\frac{1}{2\theta_0}\|\phi_x\|_{{}_{L^2}}^2+\frac{\theta_0}{2}\|u_x\|_{{}_{L^2}}^2
			+\tau \|u\|_{{}_{L^2}}\|v\|_{{}_{L^2}}+\theta_0\tau\|v\|_{{}_{L^2}}^2,
	\end{aligned}
\end{equation*}
again thanks to Young's inequality, so that
\begin{equation}\label{uprime}
	\|u_x\|_{{}_{L^2}}^2
		\leq C\left(\|\psi\|_{{}_{L^2}}^2+\|\phi_x\|_{{}_{L^2}}^2
			+\|u\|_{{}_{L^2}}^2+\|v\|_{{}_{L^2}}^2\right),
\end{equation}
for some strictly positive constant depending on $\tau$ and $\theta_0$.
Hence, from the first estimate in \eqref{forvuprime},
we deduce for $\Re\lambda\geq 0$ and $|\Im\lambda|\geq \theta_0>0$, that
\begin{equation*}
	\begin{aligned}
	\frac12\,b_0\|v\|_{{}_{L^2}}^2
		&\leq  \frac{1}{b_0}\|\psi\|_{{}_{L^2}}^2+\frac{\tau^2}{b_0}\|u\|_{{}_{L^2}}^2
			+\|\phi_x\|_{{}_{L^2}}\|u_x\|_{{}_{L^2}}\\
		&\leq  \frac{1}{b_0}\|\psi\|_{{}_{L^2}}^2+\frac{\tau^2}{b_0}\|u\|_{{}_{L^2}}^2
			+\frac{1}{2\eta}\|\phi_x\|_{{}_{L^2}}^2+\frac{\eta}{2}\|u_x\|_{{}_{L^2}}^2,
	\end{aligned}
\end{equation*}
for any $\eta>0$. By choosing $\eta$ sufficiently small and taking advantage of \eqref{uprime}, 
we deduce
\begin{equation*}
	\|v\|_{{}_{L^2}}^2
		\leq C\left(\|\psi\|_{{}_{L^2}}^2+\|\phi_x\|_{{}_{L^2}}^2+\|u\|_{{}_{L^2}}^2\right),
\end{equation*}
for some strictly positive constant $C>0$ depending on $\tau, b_0$ and $\theta_0$.
\qed
\end{proof}

Thanks to Lemma \ref{lemma:vuprime}, it is enough to estimate $u$ in $L^2$.
To this aim, we state and prove the following elementary result.

\begin{lemma}
Let $0\leq A\leq B$ with $B>0$.
Given $c_0>0$, set $\Sigma_0:=\{(x,y)\,:\,x\geq 0, |y|\geq c_0\}$. Then
\begin{equation}\label{somesup}
	\sup_{(x,y)\in\Sigma_0} \frac{1+A\sqrt{x^2+y^2}}{\bigl(1+B\,x\bigr)|y|}
	\leq A+\frac{1}{y_0}.
\end{equation}
\end{lemma}

\begin{proof}
\smartqed
Fix $c_0>0$ and $y\geq c_0$ and consider $y$ such that $|y|\geq y_0$
We want to prove that $M$ is such that
\begin{equation*}
	F(x):=M\bigl(1+B\,x\bigr)|y|-A\sqrt{x^2+y^2}\geq 1, \qquad\qquad \forall x\geq 0.
\end{equation*}
Since the function $F$ is concave, it is enough to require that the condition $F(x)\geq 1$
is satisfied at $x=0$ and at $x=+\infty$. 
The former condition is satisfied if  $M\geq A+1/y_0$;
the latter, if  $M>A/By_0$.
Since $B>A$, the first condition implies the second.
\qed
\end{proof}

\begin{lemma}\label{lemma:uell2}
Let $0<b_0\leq b(x)\leq b_1$ for any $x\in\R$.
Given $(\phi,\psi)\in \cX = H^1 \times L^2$, let $(u,v)\in \cD = H^2 \times H^1$
be such that \eqref{reszerospeed} holds.
Then there exists $M>0$ such that for any $\theta_0>0$, there exists a constant
$C>0$ (depending on $\tau, b_0, M$ and $\theta_0$) such that
\begin{equation}\label{uell2}
	\|u\|_{{}_{L^2}}\leq C\left(\|\phi\|_{{}_{L^2}}+\|\psi\|_{{}_{L^2}}\right),
\end{equation}
for any $\lambda$ with either $\Re\lambda\geq M$ or $|\Im\lambda|\geq \theta_0>0$.
\end{lemma}

\begin{proof}
\smartqed
Multiplying the second equation by $\bar u$ we obtain
\begin{equation*}
	\bigl(\tau\lambda^2 +\lambda b(x) + \tau \bigr)|u|^2+|\bar{u}_x|^2-(u_x \bar{u})_x
		=(b(x)+\tau\lambda)\bar u\phi+\bar u\psi.
\end{equation*}
Integrating in $\R$ and taking real and imaginary parts, we infer
\begin{equation}\label{realandimag}
	\begin{aligned}
	\bigl(\tau(\Re\lambda)^2-\tau(\Im\lambda)^2
				+b_0\Re\lambda + \tau\bigr)\|u\|_{{}_{L^2}}
		&\leq (b_1+\tau|\lambda|)\|\phi\|_{{}_{L^2}}+\|\psi\|_{{}_{L^2}},\\
	|\Im\lambda|\bigl(2\tau\Re\lambda+b_0\bigr)\|u\|_{{}_{L^2}}
		&\leq (b_1+\tau|\lambda|)\|\phi\|_{{}_{L^2}}+\|\psi\|_{{}_{L^2}},
	\end{aligned}
\end{equation}
Applying \eqref{somesup}, from the second inequality in \eqref{realandimag}, we infer
\begin{equation}\label{estimate1}
	\|u\|_{{}_{L^2}}
		\leq \frac{1}{b_0}\left(\tau+\frac{b_1}{\theta_0}\right)\|\phi\|_{{}_{L^2}}
			 +\frac{1}{b_0\,\theta_0}\,\|\psi\|_{{}_{L^2}}
\end{equation}
for any $\lambda$ such that $|\Im\lambda|\geq \theta_0>0$ and $\Re\lambda\geq 0$.

Using relations \eqref{realandimag}, we deduce
\begin{equation*}
	\begin{aligned}
	\bigl(\tau(\Re\lambda)^2 +b_0\Re\lambda + \tau\bigr)\|u\|_{{}_{L^2}}
	&\leq (b_1+\tau|\lambda|)\|\phi\|_{{}_{L^2}}+\|\psi\|_{{}_{L^2}}
		+\tau(\Im\lambda)^2\|u\|_{{}_{L^2}}\\
	&\leq \frac{2\tau\Re\lambda+\tau|\Im\lambda|+b_0}{2\tau\Re\lambda+b_0}
			\left((b_1+\tau|\lambda|)\|\phi\|_{{}_{L^2}}+\|\psi\|_{{}_{L^2}}\right).
	\end{aligned}
\end{equation*}
For $\Re\lambda$ large and $|\Im\lambda|\leq m_0|\Re\lambda|$, there holds
\begin{equation*}
	\|u\|_{{}_{L^2}}\leq \frac{C}{\Re\lambda}\left(\|\phi\|_{{}_{L^2}}
		+\frac{1}{\Re\lambda}\|\psi\|_{{}_{L^2}}\right)
\end{equation*}
for some strictly positive constant $C>0$.
\qed
\end{proof}

Collecting the statements contained in Lemma \ref{lemma:vuprime} and Lemma \ref{lemma:uell2},
we deduce the following result.

\begin{proposition}
Given $0<b_0\leq b(x)\leq b_1$ for any $x\in\R$,
let $\cL_0$ be the operator defined in \eqref{defopc0} on the space $\cX = H^1 \times L^2$ with
dense domain $\cD = H^2 \times H^1$. Then,\par
{(i) } there exists $M>0$ such that
\begin{equation*}
	\{\lambda \in \C  \,:\,\Re\lambda\geq 0\}\setminus[0,M]\subseteq\rho(\cL_0),
\end{equation*}
where $\rho(\cL_0)$ is the resolvent set of $\cL_0$; and, \par
{(ii) } for any $\theta_0>0$, there exists a constant $C>0$ for which
\[
\|(\lambda-\cL_0)^{-1}\|\leq C
\] 
for any $\lambda$ such that either $\Re\lambda\geq M$ or $|\Im\lambda|\geq \theta_0>0$.
\end{proposition}

\def\cprime{$'$}

%
%
%
%
%

\begin{thebibliography}{10}

\bibitem{AGJ90}
{\sc J.~Alexander, R.~Gardner, and C.~K. R.~T. Jones}, {\em A topological
  invariant arising in the stability analysis of travelling waves}, J. Reine
  Angew. Math. \textbf{410} (1990), pp.~167--212.

\bibitem{AlCa79}
{\sc S.~M. Allen and J.~W. Cahn}, {\em A microscopic theory for antiphase
  boundary motion and its application to antiphase domain coarsening}, Acta
  Metallurgica \textbf{27} (1979), no.~6, pp.~1085--1095.

\bibitem{BrJoK14}
{\sc J.~C. Bronski, M.~A. Johnson, and T.~Kapitula}, {\em An instability index
  theory for quadratic pencils and applications}, Comm. Math. Phys.
  \textbf{327} (2014), no.~2, pp.~521--550.

\bibitem{Cop78}
{\sc W.~A. Coppel}, {\em Dichotomies in Stability Theory}, no.~629 in Lecture
  Notes in Mathematics, Springer-Verlag, New York, 1978.

\bibitem{Eng2}
{\sc H.~Engler}, {\em Relations between travelling wave solutions of
  quasilinear parabolic equations}, Proc. Amer. Math. Soc. \textbf{93} (1985),
  no.~2, pp.~297--302.

\bibitem{Fol17}
{\sc R.~Folino}, {\em Slow motion for a hyperbolic variation of {A}llen-{C}ahn
  equation in one space dimension}, J. Hyperbolic Differ. Equ. \textbf{14}
  (2017), no.~1, pp.~1--26.

\bibitem{FLM17}
{\sc R.~Folino, C.~Lattanzio, and C.~Mascia}, {\em Metastable dynamics for
  hyperbolic variations of the {A}llen-{C}ahn equation}, Commun. Math. Sci.
  \textbf{15} (2017), no.~7, pp.~2055--2085.

\bibitem{GZ98}
{\sc R.~A. Gardner and K.~Zumbrun}, {\em The gap lemma and geometric criteria
  for instability of viscous shock profiles}, Comm. Pure Appl. Math.
  \textbf{51} (1998), no.~7, pp.~797--855.

\bibitem{GiKe15}
{\sc B.~H. Gilding and R.~Kersner}, {\em On a nonlinear hyperbolic equation
  with a bistable reaction term}, Nonlinear Anal. \textbf{114} (2015),
  pp.~169--185.

\bibitem{Had1}
{\sc K.~P. Hadeler}, {\em Hyperbolic travelling fronts}, Proc. Edinburgh Math.
  Soc. (2) \textbf{31} (1988), no.~1, pp.~89--97.

\bibitem{Hame99}
{\sc F.~Hamel}, {\em Formules min-max pour les vitesses d'ondes progressives
  multidimensionnelles}, Ann. Fac. Sci. Toulouse Math. (6) \textbf{8} (1999),
  no.~2, pp.~259--280.

\bibitem{HaeMa1}
{\sc J.~H{\"a}rterich and C.~Mascia}, {\em Front formation and motion in
  quasilinear parabolic equations}, J. Math. Anal. Appl. \textbf{307} (2005),
  no.~2, pp.~395--414.

\bibitem{Holm1}
{\sc E.~E. Holmes}, {\em Is diffusion too simple? {C}omparisons with a
  telegraph model of dispersal}, American Naturalist \textbf{142} (1993),
  no.~5, pp.~779--796.

\bibitem{IDRWZB95}
{\sc G.~Iz{\'u}s, R.~Deza, O.~Ram{\'{\i}}rez, H.~S. Wio, D.~H. Zanette, and
  C.~Borzi}, {\em Global stability of stationary patterns in bistable
  reaction-diffusion systems}, Phys. Rev. E (3) \textbf{52} (1995), no.~1, part
  A, pp.~129--136.

\bibitem{KaPro13}
{\sc T.~Kapitula and K.~Promislow}, {\em Spectral and dynamical stability of
  nonlinear waves}, vol.~185 of Applied Mathematical Sciences, Springer, New
  York, 2013.

\bibitem{KS98}
{\sc T.~Kapitula and B.~Sandstede}, {\em Stability of bright solitary-wave
  solutions to perturbed nonlinear {S}chr\"odinger equations}, Phys. D
  \textbf{124} (1998), no.~1-3, pp.~58--103.

\bibitem{Kat80}
{\sc T.~Kato}, {\em Perturbation Theory for Linear Operators}, Classics in
  Mathematics, Springer-{V}erlag, {N}ew {Y}ork, {S}econd~ed., 1980.

\bibitem{KoMi14}
{\sc R.~Koll{\'a}r and P.~D. Miller}, {\em Graphical {K}rein signature theory
  and {E}vans-{K}rein functions}, SIAM Rev. \textbf{56} (2014), no.~1,
  pp.~73--123.

\bibitem{LMPS16}
{\sc C.~Lattanzio, C.~Mascia, R.~G. Plaza, and C.~Simeoni}, {\em Analytical and
  numerical investigation of traveling waves for the {A}llen-{C}ahn model with
  relaxation}, Math. Models Methods Appl. Sci. \textbf{26} (2016), no.~5,
  pp.~931--985.

\bibitem{LaSu1}
{\sc Y.~Latushkin and A.~Sukhtayev}, {\em The algebraic multiplicity of
  eigenvalues and the {E}vans function revisited}, Math. Model. Nat. Phenom.
  \textbf{5} (2010), no.~4, pp.~269--292.

\bibitem{Lbr67a}
{\sc H.~M. Lieberstein}, {\em On the {H}odgkin-{H}uxley partial differential
  equation}, Math. Biosci. \textbf{1} (1967), no.~1, pp.~45--69.

\bibitem{Markus88}
{\sc A.~S. Markus}, {\em Introduction to the spectral theory of polynomial
  operator pencils}, vol.~71 of Translations of Mathematical Monographs,
  American Mathematical Society, Providence, RI, 1988.

\bibitem{MZ02}
{\sc C.~Mascia and K.~Zumbrun}, {\em Pointwise {G}reen's function bounds and
  stability of relaxation shocks}, Indiana Univ. Math. J. \textbf{51} (2002),
  no.~4, pp.~773--904.

\bibitem{McKe70}
{\sc H.~P. McKean, Jr.}, {\em Nagumo's equation}, Advances in Math. \textbf{4}
  (1970), pp.~209--223.

\bibitem{MFF1}
{\sc V.~M{\'e}ndez, J.~Fort, and J.~Farjas}, {\em Speed of wave-front solutions
  to hyperbolic reaction-diffusion equations}, Phys. Rev. E (3) \textbf{60}
  (1999), no.~5, part A, pp.~5231--5243.

\bibitem{Mik94}
{\sc A.~S. Mikha{\u\i}lov}, {\em Foundations of synergetics {I}. Distributed
  active systems}, vol.~51 of Springer Series in Synergetics, Springer-Verlag,
  Berlin, second~ed., 1994.

\bibitem{MurI3ed}
{\sc J.~D. Murray}, {\em Mathematical biology {I}. An introduction}, vol.~17 of
  Interdisciplinary Applied Mathematics, Springer-Verlag, New York, third~ed.,
  2002.

\bibitem{NAY62}
{\sc J.~Nagumo, S.~Arimoto, and S.~Yoshizawa}, {\em An active pulse
  transmission line simulating nerve axon}, Proc. IRE \textbf{50} (1962),
  no.~10, pp.~2061--2070.

\bibitem{Pal1}
{\sc K.~J. Palmer}, {\em Exponential dichotomies and transversal homoclinic
  points}, J. Differential Equations \textbf{55} (1984), no.~2, pp.~225--256.

\bibitem{Pal2}
{\sc K.~J. Palmer}, {\em Exponential
  dichotomies and {F}redholm operators}, Proc. Amer. Math. Soc. \textbf{104}
  (1988), no.~1, pp.~149--156.

\bibitem{Pazy83}
{\sc A.~Pazy}, {\em Semigroups of linear operators and applications to partial
  differential equations}, vol.~44 of Applied Mathematical Sciences,
  Springer-Verlag, New York, 1983.

\bibitem{Per1}
{\sc L.~Perko}, {\em Rotated vector fields}, J. Differential Equations
  \textbf{103} (1993), no.~1, pp.~127--145.

\bibitem{ProWe84}
{\sc M.~H. Protter and H.~F. Weinberger}, {\em Maximum principles in
  differential equations}, Springer-Verlag, New York, 1984.
\newblock Corrected reprint of the 1967 original.

\bibitem{ReRo04}
{\sc M.~Renardy and R.~C. Rogers}, {\em An introduction to partial differential
  equations}, vol.~13 of Texts in Applied Mathematics, Springer-Verlag, New
  York, second~ed., 2004.

\bibitem{Rott11}
{\sc J.~Rottmann-Matthes}, {\em Linear stability of traveling waves in
  first-order hyperbolic {PDE}s}, J. Dynam. Differential Equations \textbf{23}
  (2011), no.~2, pp.~365--393.

\bibitem{Rott12a}
{\sc J.~Rottmann-Matthes}, {\em Stability and
  freezing of nonlinear waves in first order hyperbolic {PDE}s}, J. Dynam.
  Differential Equations \textbf{24} (2012), no.~2, pp.~341--367.

\bibitem{San02}
{\sc B.~Sandstede}, {\em Stability of travelling waves}, in Handbook of
  dynamical systems, Vol. 2, B.~Fiedler, ed., North-Holland, Amsterdam, 2002,
  pp.~983--1055.

\bibitem{We10}
{\sc H.~Weyl}, {\em \"{U}ber gew\"ohnliche {D}ifferentialgleichungen mit
  {S}ingularit\"aten und die zugeh\"origen {E}ntwicklungen willk\"urlicher
  {F}unktionen}, Math. Ann. \textbf{68} (1910), no.~2, pp.~220--269.

\bibitem{Zum04}
{\sc K.~Zumbrun}, {\em Stability of large-amplitude shock waves of compressible
  {N}avier-{S}tokes equations}, in Handbook of mathematical fluid dynamics.
  Vol. III, S.~Friedlander and D.~Serre, eds., North-Holland, Amsterdam, 2004,
  pp.~311--533.

\end{thebibliography}
%

\end{document}